\RequirePackage{fix-cm}
\documentclass[smallextended]{svjour3}       
\smartqed  
\usepackage{graphicx}
\usepackage{multirow}
\usepackage{booktabs}
\usepackage{algpseudocode}
\usepackage{subcaption}
\usepackage{algorithmicx,algorithm}
\usepackage{amsmath,amssymb}
\newtheorem{assumption}{Assumption}
\usepackage{xcolor}
\usepackage[justification=centering]{caption}
\usepackage{threeparttable} 
\usepackage[misc]{ifsym}
\usepackage{hyperref}
\usepackage{enumitem}

\newcommand{\cM}{{\cal M}}
\newcommand{\cT}{{\cal T}}

\newcommand{\hw}{{\hat w}}
\newcommand{\by}{{\bar y}}
\newcommand{\bz}{{\bar z}}
\newcommand{\bx}{{\bar x}}
\newcommand{\bw}{{\bar w}}

\begin{document}

\title{HPR-LP: An implementation of an HPR method for solving linear programming}



\author{Kaihuang Chen \and Defeng Sun \and Yancheng Yuan \and Guojun Zhang \and Xinyuan Zhao      
}


\institute{Kaihuang Chen \at
              Department of Applied Mathematics, The Hong Kong Polytechnic University, Hung Hom, Hong Kong
              \\
              \email{kaihuang.chen@connect.polyu.hk}           
           \and
               Defeng Sun (\Letter) \at
              Department of Applied Mathematics,  The Hong Kong Polytechnic University, Hung Hom, Hong Kong
              \\
              \email{defeng.sun@polyu.edu.hk}
        \and      Yancheng Yuan \at
              Department of Data Science and Artificial Intelligence,  The Hong Kong Polytechnic University, Hung Hom, Hong Kong \\
              \email{yancheng.yuan@polyu.edu.hk}
        \and
            Guojun Zhang \at
              Department of Applied Mathematics, The Hong Kong Polytechnic University, Hung Hom, Hong Kong \\
              \email{guojun.zhang@connect.polyu.hk}
         \and
            Xinyuan Zhao \at
              Department of Mathematics, Beijing University of Technology, Beijing, P.R. China \\
              \email{xyzhao@bjut.edu.cn}
}

\maketitle

\begin{abstract}
In this paper, we introduce an HPR-LP solver, an implementation of a Halpern Peaceman-Rachford (HPR) method with semi-proximal terms for solving linear programming (LP). The HPR method enjoys the iteration complexity of \(O(1/k)\) in terms of the Karush-Kuhn-Tucker residual and the objective error. Based on the complexity results, we design an adaptive strategy of restart and penalty parameter update to improve the efficiency and robustness of the HPR method. We conduct extensive numerical experiments on different LP benchmark datasets using NVIDIA A100-SXM4-80GB GPU in different stopping tolerances. Our solver's Julia version achieves a \textbf{2.39x} to \textbf{5.70x} speedup measured by SGM10 on benchmark datasets with presolve (\textbf{2.03x} to \textbf{4.06x} without presolve) over the award-winning solver PDLP with the tolerance of $10^{-8}$.

\end{abstract}

\section{Introduction}
\label{sec:1}
In this paper, we introduce how to implement a Halpern Peaceman-Rachford (HPR) method with semi-proximal terms \cite{sun2024accelerating} to solve the following linear programming (LP) problems:
\begin{equation}\label{model:primalLP}
\begin{aligned}
\min _{x \in \mathbb{R}^n} &\quad  \langle c, x \rangle \\
\text { s.t. } & A_1 x = b_1 \\
& A_2 x \geq b_2 \\
& x \in C,
\end{aligned}
\end{equation}
where \(A_1 \in \mathbb{R}^{m_1 \times n}\), \(A_2 \in \mathbb{R}^{m_2 \times n}\), \(b_1 \in \mathbb{R}^{m_1}\), \(b_2 \in \mathbb{R}^{m_2}\), \(c \in \mathbb{R}^{n}\), and \(C := \{x \in \mathbb{R}^n \mid l \leq x \leq u\}\) with given vectors \(l \in (\mathbb{R} \cup \{-\infty\})^n\) and \(u \in (\mathbb{R} \cup \{+\infty\})^n\). Let \(A = [A_1; A_2] \in \mathbb{R}^{m \times n}\) with \(m = m_1 + m_2\), and \(b = [b_1; b_2] \in \mathbb{R}^m\).  In this paper, we assume that $A$ is a non-zero matrix.  Then, the dual of problem \eqref{model:primalLP} is given by
\begin{equation}\label{model:dualLP}
\begin{aligned}
\min _{y \in \mathbb{R}^m, z \in \mathbb{R}^n} & -\langle b, y \rangle +  \delta_{D}(y)+
  \delta_{C}^{*}(-z) \\
\text { s.t. } & A^{*} y + z = c,
\end{aligned}
\end{equation}
where $\delta_{D}(\cdot)$ is the indicator function over \(D:= \{ y = (y_1, y_2) \in \mathbb{R}^{m_1} \times \mathbb{R}^{m_2}_{+}\}\) and $\delta_{C}^{*}(\cdot)$ is the conjugate of $\delta_{C}(\cdot)$.

Traditionally, the commercial solvers for LP \cite{gurobi,cplex2009v12} are based on the simplex methods and interior point methods. In recent years, first-order methods (FOMs) have gained increasing attention for solving large-scale LP problems due to their low iteration cost and ease of parallelization \cite{applegate2021practical,applegate2023faster,basu2020eclipse,deng2024enhanced,lin2021admm,lu2024restarted,o2021operator,o2016conic,stellato2020osqp}. Notably, Applegate et al. \cite{applegate2021practical,applegate2023faster} developed an award-winning solver PDLP\footnote{The authors of \cite{applegate2021practical,applegate2023faster} were awarded the Beale–Orchard-Hays Prize for Excellence in Computational Mathematical Programming at the 25th International Symposium on Mathematical Programming (\href{https://ismp2024.gerad.ca/}{https://ismp2024.gerad.ca/}), July 21-26, 2024, Montréal, Canada.}, which employed the primal-dual hybrid gradient (PDHG) method \cite{zhu2008efficient} as its base algorithm, equipped with several effective implementation techniques, including the use of the ergodic iterate as a restart point\footnote{In the GPU implementation (e.g., \cite{lu2023cupdlp} and \cite{lu2023cupdlp_c}), the restart point is selected based on the weighted KKT residual, utilizing either the last iterate or the ergodic iterate. For theoretical analysis, the ergodic iterate is used as the restart point \cite{applegate2023faster}.}, an update rule for the penalty parameter \(\sigma\), and a line search strategy. The GPU implementation of PDLP (cuPDLP.jl \cite{lu2023cupdlp} and cuPDLP-c \cite{lu2023cupdlp_c}) has shown some advantages over commercial LP solvers like Gurobi \cite{gurobi} and COPT \cite{ge2022cardinal} for solving large-scale LP problems. It is worth noting that the ergodic sequence of the semi-proximal ADMM (sPADMM) \cite{fazel2013hankel}, including PDHG \cite{chambolle2011first,esser2010general}, achieves an $O(1/k)$ iteration complexity with respect to the objective error and the feasibility violation \cite{cui2016convergence}, where $k$ is the iteration number. To the best of our knowledge, it is still unknown whether the ergodic sequence of sPADMM can achieve an $O(1/k)$ iteration complexity in terms of the Karush-Kuhn-Tucker (KKT) residual for LP.\footnote{Shen and Pan \cite{shen2016weighted} extended Monteiro and Svaiter's ergodic \(O(1/k)\) iteration complexity result with respect to the $\varepsilon$-subdifferential residual \cite{monteiro2013iteration} from ADMM to sPADMM, which can yield an \(O(1/\sqrt{k})\) iteration complexity for the KKT residual of LP problems.}

Recently, some progress in complexity results has been achieved in accelerating the preconditioned (semi-proximal) PR \cite{sun2024accelerating,yang2025accelerated,zhang2022efficient} using the Halpern iteration \cite{halpern1967fixed,lieder2021convergence,sabach2017first}. In particular, Zhang et al. \cite{zhang2022efficient} applied the Halpern iteration to the PR splitting method \cite{eckstein1992douglas,lions1979splitting}, developing the HPR method without proximal terms.\footnote{The HPR method without proximal terms is equivalent to Kim's accelerated proximal point method (PPM) applied to the Douglas-Rachford method in view of the relationship between Halpern's iteration and Kim's accelerated PPM \cite{kim2021accelerated,ryu2022large}.} This HPR method achieves an iteration complexity of \(O(1/k)\) for the KKT residual and the objective error.
Subsequently, Yang et al. \cite{yang2025accelerated} reformulated the proximal PR method as a proximal point method (PPM) with a positive definite preconditioner, developing an HPR method with proximal terms via the Halpern iteration that obtains the \(O(1/k)\) iteration complexity for the weighted fixed-point residual. Finally, Sun et al. \cite{sun2024accelerating} reformulated the semi-proximal PR method into a degenerate PPM (dPPM) with a positive semidefinite preconditioner \cite{bredies2022degenerate} and applied the Halpern iteration to this dPPM, resulting in an HPR method with semi-proximal terms that enjoys the desired \(O(1/k)\) iteration complexity for the KKT residual and the objective error.\footnote{The Halpern PDHG for LP proposed in \cite{lu2024restarted} corresponds to a special case of the HPR method in \cite{sun2024accelerating} with one proximal term to be strongly convex and $\rho=1$. {Inspired by our work, Lu and Yang \cite{lu2024restarted}, in their revised version, extended their approach by introducing the reflected restarted Halpern PDHG ($r^2$HPDHG), which adopts $\rho=2$ within the framework of the HPR method in \cite{sun2024accelerating}. By incorporating similar algorithmic enhancements as cuPDLP \cite{lu2023cupdlp}, $r^2$HPDHG achieves a 1.27x to 1.33x speedup over cuPDLP in their numerical experiments.}} Compared to existing algorithms for solving LP, the HPR method \cite{sun2024accelerating,zhang2022efficient} offers a key theoretical advantage: an \(O(1/k)\) iteration complexity in terms of the KKT residual, which motivates us to explore the potential of HPR for solving large-scale LP problems.

The main purpose of this paper is to introduce a solver called HPR-LP for solving LP problems. The main features of our solver are highlighted below.
\begin{enumerate}
\item Based on the iteration complexity of $O(1/k)$ in terms of the KKT residual, in our HPR-LP solver we integrate a restart strategy and an update rule of penalty parameter $\sigma$  into the HPR method with semi-proximal terms for solving large-scale LP problems.
\item We test the numerical performance of the HPR-LP solver on different LP benchmark datasets using NVIDIA A100-SXM4-80GB GPU in different stopping tolerances. The Julia version of our solver achieves a \textbf{2.39x} to \textbf{5.70x} speedup measured by SGM10 on benchmark datasets with presolve (\textbf{2.03x} to \textbf{4.06x} without presolve) over the award-winning solver PDLP with the tolerance of $10^{-8}$. Moreover, for the same accuracy, HPR-LP successfully solves more problems than cuPDLP.jl does.
\end{enumerate}

The remaining parts of this paper are organized as follows. In Section \ref{sec:2}, we briefly introduce the base algorithm, an HPR method with semi-proximal terms, for solving LP. Then, in Section \ref{sec:3}, we discuss the implementation of HPR-LP, which incorporates a restart strategy and an update rule for the penalty parameter. Section \ref{sec:4} provides extensive numerical results on different LP benchmark sets. Finally, we conclude the paper in Section \ref{sec:5}.
\paragraph{Notation.}
Let \(\mathbb{R}^n\) be the \(n\)-dimensional real Euclidean space equipped with an inner product \(\langle \cdot, \cdot \rangle\) and its induced norm \(\|\cdot\|\). We denote the nonnegative orthant of \(\mathbb{R}^n\) as \(\mathbb{R}^n_{+}\). For a matrix \(A \in \mathbb{R}^{m \times n}\), we denote its transpose by \(A^{*}\) and its spectral norm by \(\|A\| := \sqrt{\lambda_{1}\left(AA^{*}\right)}\), where \(\lambda_{1}\left(AA^{*}\right)\) represents the largest eigenvalue of the symmetric matrix \(AA^{*}\). Additionally, for any self-adjoint positive semidefinite linear operator $\mathcal{M}: \mathbb{R}^n \to \mathbb{R}^{n}$, we define the semi-norm $\|x\|_{\cM}:= \sqrt{\langle x,\cM x\rangle}$ for any $x\in \mathbb{R}^{n}$. For any convex function \(f: \mathbb{R}^n \to (-\infty, +\infty]\), we denote its effective domain as \(\operatorname{dom}(f) := \{x \in \mathbb{R}^n \mid f(x) < +\infty\}\), its subdifferential as $\partial f(\cdot)$, its conjugate as \(f^*(z) := \sup_{x \in \mathbb{R}^n} \{\langle x, z \rangle - f(x)\}, \, z \in \mathbb{R}^n\), and its proximal mapping as \(\operatorname{Prox}_f(x) := \arg \min_{z \in \mathbb{R}^n} \left\{ f(z) + \frac{1}{2} \|z - x\|^2 \right\}, \, x \in \mathbb{R}^n\). Let \(C \subseteq \mathbb{R}^n\) be a convex set. Denote the indicator function over \(C\) by \(\delta_C(\cdot)\), i.e., for any \(x \in \mathbb{R}^n\), \(\delta_C(x) = 0\) if \(x \in C\), and \(\delta_C(x) = +\infty\) if \(x \notin C\).
We write the distance of \(x \in \mathbb{R}^n\) to \(C\) as \(\operatorname{dist}(x, C) := \inf_{z \in C} \|z - x\|\). For a given closed convex set \(C \subseteq \mathbb{R}^n\) and a given point \(x \in \mathbb{R}^n\), the Euclidean projection of \(x\) onto \(C\) is denoted by \(\Pi_C(x) := \arg \min \{\|x - z\| \mid z \in C\}\). Moreover, for any \(x \in C\), we use \(\mathcal{N}_C(x)\) to denote the normal cone of \(C\) at \(x\).

\section{Preliminaries}
\label{sec:2}
In this section, we introduce an HPR method with semi-proximal terms for solving LP problems. According to \cite[Corollary 28.3.1]{rockafellar1970convex}, we know that $({y}^*, {z}^*) \in$ $\mathbb{R}^{m} \times \mathbb{R}^{n}$ is an optimal solution to problem \eqref{model:dualLP} if and only if there exists ${x}^* \in \mathbb{R}^{n}$ such that $({y}^*, {z}^*,{x}^*)$ satisfies the following KKT system:
	\begin{equation}\label{eq:KKT}
		0\in Ax^*-b+ \mathcal{N}_{D}(y^*), \quad  0 \in z^{*} +\mathcal{N}_{C}(x^{*}) , \quad  A^{*}{y}^*+{z}^*-c=0.
	\end{equation}
 For any $(y,z,x)\in \mathbb{R}^{m} \times \mathbb{R}^{n} \times \mathbb{R}^{n}$, the augmented Lagrangian function associated with the dual problem \eqref{model:dualLP} is defined as
$$
L_{\sigma}(y,z;x):= -\langle b, y \rangle+  \delta_{D}(y) + \delta_{C}^{*}(-z)+\langle x, A^{*}y+z-c \rangle +\frac{\sigma}{2}\|A^{*}y+z-c\|^2,
$$
where $\sigma>0$ is a given penalty parameter. For notational convenience, let $w:=(y,z,x)\in \mathbb{R}^{m} \times \mathbb{R}^{n} \times \mathbb{R}^{n}$. Then, an HPR method with semi-proximal terms proposed in \cite{sun2024accelerating} for solving problems \eqref{model:primalLP} and \eqref{model:dualLP} is presented in Algorithm \ref{alg:sp-HPR}.
\begin{algorithm}[ht!]
		\caption{An HPR method  with semi-proximal terms for the problem \eqref{model:dualLP}}
		\label{alg:sp-HPR}
		\begin{algorithmic}[1]
			\State {\textbf{Input:} Set the penalty parameter \(\sigma > 0\). Let \(\mathcal{T}_1: \mathbb{R}^m \to \mathbb{R}^m\) be a self-adjoint positive semidefinite linear operator such that \(\mathcal{T}_1 + \sigma AA^{*}\) is positive definite. {Denote $w=(y,z,x)$ and $\bw=(\by,\bz,\bx)$.}
            Choose an initial point \(w^0 = (y^0, z^0, x^0) \in D \times \mathbb{R}^n \times \mathbb{R}^n\).\vspace{3pt}}
               \For{$k=0,1,...,$ \vspace{3pt}}
			\State {Step 1. $\displaystyle \bz^{k+1}=\underset{z \in \mathbb{R}^{n}}{\arg \min }\left\{L_{\sigma}\left(y^k, z ; x^k\right)\right\}$;}
			\State{Step 2. $ \displaystyle \bx^{k+1}={x}^k+\sigma (A^{*}{y}^{k}+\bz^{k+1}-c) $;}
			\State {Step 3. $\displaystyle \by^{k+1}=\underset{y \in \mathbb{R}^{m}}{\arg \min }\left\{L_{\sigma}\left(y, \bz^{k+1} ; \bx^{k+1}\right)+\frac{\sigma}{2}\|y-y^{k}\|_{\mathcal{T}_1}^2\right \}$;}
			\State {Step 4. $\displaystyle \hw^{k+1}= 2\bw^{k+1}-{w}^{k} $;}

			\State {Step 5.  $\displaystyle w^{k+1}=\frac{1}{k+2} w^{0}+\frac{k+1}{k+2}\hw^{k+1}$; \vspace{3pt}}
               \EndFor \vspace{3pt}
             \State{\textbf{Output:} Iteration sequence $\{\bw^k\}$.}
		\end{algorithmic}
	\end{algorithm}

Algorithm \ref{alg:sp-HPR} corresponds to the accelerated preconditioned (semi-proximal) ADMM (pADMM)  introduced in \cite{sun2024accelerating} with \(\alpha = 2\), where Step 5 is the Halpern iteration with a stepsize of \(\frac{1}{k+2}\) \cite{halpern1967fixed,lieder2021convergence}. We refer to it as an HPR method with semi-proximal terms because, for the convex optimization problem \eqref{primal}, the Halpern accelerating pADMM (without proximal terms) \cite{sun2024accelerating} is equivalent to the HPR method without proximal terms \cite{zhang2022efficient}, whose proof can be found in Appendix \ref{appendix-A}. Now, to discuss the global convergence of the HPR method  with semi-proximal terms presented in Algorithm \ref{alg:sp-HPR}, we make the following assumption:
\begin{assumption}\label{ass: CQ}
There exists a vector $({y}^*, {z}^*,{x}^*)\in \mathbb{R}^{m}\times \mathbb{R}^{n}\times \mathbb{R}^{n}$ satisfying the KKT system \eqref{eq:KKT}.
\end{assumption}
Under Assumption \ref{ass: CQ}, solving problems
\eqref{model:primalLP} and \eqref{model:dualLP} is equivalent to finding a $w^{*} \in \mathbb{R}^{m} \times \mathbb{R}^{n} \times \mathbb{R}^{n}$ such that $0\in \cT w^{*}$, where the maximal monotone operator $\cT$ is defined by
	\begin{equation}\label{def:T}
		\cT w=\left(\begin{array}{c}
			-b+\mathcal{N}_{D}(y) +A {x}  \\
			-\partial \delta_{C}^{*}(-z) +  x  \\
			c-A^{*}{y}-{z} \\
		\end{array}     \right),  \quad  \forall w=(y,z,x)\in \mathbb{R}^{m} \times \mathbb{R}^{n} \times \mathbb{R}^{n}.
	\end{equation}
 Consider the following self-adjoint linear operator $\cM: \mathbb{R}^{m}\times\mathbb{R}^{n}\times\mathbb{R}^{n}   \rightarrow \mathbb{R}^{m}\times\mathbb{R}^{n}\times\mathbb{R}^{n}$,
	\begin{equation}\label{def:M}
		\cM=\left[\begin{array}{ccc}
			\sigma AA^{*}+\sigma \cT_1  &\quad  0  \quad &\quad  A \\
			0 &\quad 0 \quad  & \quad 0 \\
			A^{*} &\quad  0 \quad  & \quad \frac{1}{\sigma} I_{n}
		\end{array}\right],
\end{equation}
where $I_{n}$ denotes the identity matrix in $\mathbb{R}^{n\times n}$. We can establish the equivalence between the HPR method with semi-proximal terms and the accelerated dPPM \cite{sun2024accelerating} in the following proposition.
\begin{proposition}\label{prop:equ-acc-PADMM-acc-dPPM}
 Consider {the operators} $\cT$ defined in \eqref{def:T} and $\cM$ defined in \eqref{def:M}. Then the sequence $\left\{w^k\right\}$ generated by the HPR method  with semi-proximal terms in Algorithm \ref{alg:sp-HPR} coincides with the sequence $\left\{w^k\right\}$ generated by the following accelerated dPPM
$$
\left\{\begin{array}{ll}
    &\displaystyle \bw^{k+1}=  (\cM+\cT)^{-1}\cM w^{k},  \\[6pt]
    &\displaystyle \hw^{k+1}= 2\bw^{k+1}-{w}^{k}, \\[4pt]
    &\displaystyle w^{k+1}=\frac{1}{k+2} w^{0}+\frac{k+1}{k+2}\hw^{k+1},
\end{array}\right.
$$
with the same initial point $ w^0 \in D \times \mathbb{R}^n \times \mathbb{R}^n$.
\end{proposition}
\begin{proof}
    This result can be derived similarly from \cite[Proposition 3.2]{sun2024accelerating}. \qed
\end{proof}

By utilizing the equivalence in Proposition \ref{prop:equ-acc-PADMM-acc-dPPM}, we can establish the global convergence of the HPR method with semi-proximal terms, as presented in the following proposition.
\begin{proposition}[Corollary 3.5 in \cite{sun2024accelerating}]
Suppose that Assumption \ref{ass: CQ} holds. Then the sequence \(\{\bw^k\} = \{(\by^k, \bz^k, \bx^k)\}\) generated by the HPR method with semi-proximal terms in Algorithm \ref{alg:sp-HPR} converges to the point \(w^{*} = (y^*, z^*, x^*)\), where \((y^*, z^*)\) solves problem \eqref{model:dualLP} and \(x^*\) solves problem \eqref{model:primalLP}.
\end{proposition}
Moreover, the equivalence in Proposition \ref{prop:equ-acc-PADMM-acc-dPPM}, together with the iteration complexity of the Halpern iteration \cite{lieder2021convergence}, yields the following iteration complexity for the HPR method with semi-proximal terms:
\begin{proposition}[Proposition 2.9 in \cite{sun2024accelerating}]  \label{prop:complexity-acc-pADMM-M}
Suppose that Assumption \ref{ass: CQ} holds. Then the sequences \(\{w^k\}\) and \(\{\hat{w}^k\}\) generated by the HPR method  with semi-proximal terms in Algorithm \ref{alg:sp-HPR} satisfy
\begin{equation}
\|w^{k}-\hat{w}^{k+1}\|_{\cM} \leq \frac{2\left\|w^0-w^*\right\|_{\cM}}{k+1},\quad  \forall k \geq 0 \text{ and } w^* \in \mathcal{T}^{-1}(0).
\end{equation}
\end{proposition}
To further analyze the complexity of the HPR method with semi-proximal terms in terms of the KKT residual and the objective error, we consider the residual mapping associated with the KKT system \eqref{eq:KKT}, as introduced in \cite{han2018linear}:
        \begin{equation}\label{def:KKT_residual}
		\mathcal{R}(w) = \left(
		\begin{array}{c}
			y - \Pi_{D}(y - Ax +b)\\
			x - \Pi_{C}(x -z)\\
			c - A^{*}y   -z
		\end{array}
		\right), \quad \forall w=(y,z,x) \in \mathbb{R}^{m}\times \mathbb{R}^{n}\times \mathbb{R}^{n}.
	\end{equation}
 Additionally, let $\{(\by^k,\bz^k)\}$ be the sequence generated by Algorithm \ref{alg:sp-HPR}. We define the objective error as follows:
$$
h(\by^{k+1}, \bz^{k+1}):=-\langle b, \by^{k+1} \rangle + \delta_{C}^{*}(-\bz^{k+1}) +\langle b, y^{*} \rangle - \delta_{C}^{*}(-z^*), \quad \forall k \geq 0,
$$
where $\left(y^*, z^*\right)$ is the limit point of the sequence $\{(\by^k, \bz^k)\}$. Based on the iteration complexity presented in Proposition \ref{prop:complexity-acc-pADMM-M}, we can derive the complexity result of the HPR method with semi-proximal terms in terms of the KKT residual and the objective error, as stated in the following theorem.
 \begin{theorem}[Theorem 3.7 in \cite{sun2024accelerating}]\label{Th:complexity-acc-pADMM}
			Suppose that Assumption \ref{ass: CQ} holds. Let $\{(\by^{k},\bz^{k},\bx^{k})\}$ be the sequence generated by the HPR method with semi-proximal terms in Algorithm \ref{alg:sp-HPR}, and let $w^*=(y^*,z^*,x^*)$ be the limit point of the sequence $\{(\by^{k},\bz^{k},\bx^{k})\}$ and $R_0=\|w^{0}-w^{*}\|_{\cM}$. Then for all $k \geq 0$, we have the following iteration complexity bounds:
				\begin{equation}
					\label{eq: complexity-bound-KKT}
					\displaystyle\|\mathcal{R}(\bar w^{k+1})\| \leq \left( \frac{\sigma (\|A\|+\|\sqrt{\cT_1}\|)+1}{\sqrt{\sigma}} \right) \frac{R_0}{(k+1)}
				\end{equation}
				and
				\begin{equation}\label{eq: complexity-bound-obj}
					\begin{array}{ll}
						\displaystyle \left(\frac{-1}{\sqrt{\sigma}}\|x^*\|\right) \frac{ R_0}{(k+1)}\leq h(\by^{k+1},\bz^{k+1})  \leq  \left(3R_0  + \frac{1}{\sqrt{\sigma}}\|x^*\|\right) \
						\frac{R_0}{(k+1)}.
					\end{array}
				\end{equation}
\end{theorem}
{
\begin{remark} \label{rem:complexity}
There are several complexity results related to the KKT residual of ADMM-type methods for solving LP in the literature. In the ergodic sense, Monteiro and Svaiter \cite{monteiro2013iteration} established an ergodic \(O(1/k)\) iteration complexity for ADMM with a unit dual step size in terms of the \(\varepsilon\)-subdifferential, which, as mentioned in the introduction,  implies an ergodic \(O(1/\sqrt{k})\) iteration complexity for the KKT residual. Recently, building on the ergodic $O(1/k)$ iteration complexity of the PDHG established by Chambolle and Pock \cite{chambolle2011first,chambolle2016ergodic}, Applegate et al. \cite{applegate2023faster} derived an ergodic $O(1/k)$ iteration complexity of the PDHG for solving LP, measured by primal feasibility violation, dual feasibility violation, and the primal-dual gap, all of which can be inferred from the KKT residual. In the nonergodic sense, Cui et al. \cite{cui2016convergence} showed that a majorized ADMM with semi-proximal terms achieves a nonergodic $O(1/\sqrt{k})$ iteration complexity for the KKT residual. Compared to these existing methods, the HPR method attains a much improved   $O(1/k)$ iteration complexity in terms of the KKT residual as in \eqref{eq: complexity-bound-KKT}, providing stronger theoretical guarantees that, in turn, support practical advantages in solving large-scale LP problems efficiently. For more complexity results, please refer to \cite{sun2024accelerating} and the references therein.
\end{remark}
}

\section{A Halpern Peaceman-Rachford {solver} for solving LP}\label{sec:3}
In this section, we present an HPR-LP solver for solving large-scale LP problems, as outlined in Algorithm \ref{alg:HPR_LP}. In the HPR-LP solver, we integrate a restart strategy and adaptive updates of the penalty parameter \(\sigma\) into the HPR method with semi-proximal terms.

\begin{algorithm}[ht!]
		\caption{HPR-LP: A Halpern Peaceman-Rachford method for the problem \eqref{model:dualLP}}
		\label{alg:HPR_LP}
		\begin{algorithmic}[1]
			\State {\textbf{Input:} Let \(\mathcal{T}_1: \mathbb{R}^m \to \mathbb{R}^m\) be a self-adjoint positive semidefinite linear operator such that \(\mathcal{T}_1 +  AA^{*}\) is positive definite. {Denote $w=(y,z,x)$ and $\bw=(\by,\bz,\bx)$.} Choose an initial point \(w^{0,0} = (y^{0,0}, z^{0,0}, x^{0,0}) \in D \times \mathbb{R}^n \times \mathbb{R}^n\).\vspace{3pt} }
                \State{\textbf{Initialization}: Set the outer loop counter \(r = 0\), the total loop counter \(k = 0\), and the initial penalty parameter \(\sigma_{0} > 0\). \vspace{3pt}}
                \Repeat  \vspace{3pt}
                       \State{{initialize the inner loop: set inner loop counter $t=0$;} \vspace{3pt}}
                       \Repeat \vspace{3pt}
			\State {$\displaystyle \bz^{r,t+1}=\underset{z \in \mathbb{R}^{n}}{\arg \min }\left\{L_{\sigma_r}\left(y^{r,t}, z ; x^{r,t}\right)\right\}$; \vspace{3pt} }
			\State{$\displaystyle \bx^{r,t+1}={x}^{r,t}+{\sigma_r} (A^{*}{y}^{r,t}+\bz^{r,t+1}-c) $; \vspace{3pt}}
			\State {$\displaystyle \by^{r,t+1}=\underset{y \in \mathbb{R}^{m}}{\arg \min }\left\{L_{\sigma_r}\left(y, \bz^{r,t+1} ; \bx^{r,t+1}\right)+\frac{\sigma_{r}}{2}\|y-y^{r,t}\|_{\mathcal{T}_1}^2\right \};$ \vspace{3pt}	}
			\State {$\displaystyle \hw^{r,t+1}= 2\bw^{r,t+1}-{w}^{r,t} $; \vspace{3pt} }
			\State {$\displaystyle w^{r,t+1}=\frac{1}{t+2} w^{r,0}+\frac{t+1}{t+2}\hw^{r,t+1}$; \vspace{3pt}}
   			\State {$t=t+1$, $k=k+1$;}
                 \Until one of the restart criteria holds or termination criteria hold \vspace{3pt}
                 \State{\textbf{restart the inner loop:}
                 $\tau_{r}=t, w^{r+1,0}=\bw^{r,\tau_{r}}$ \vspace{3pt}},
                  \State{$\displaystyle \sigma_{r+1} ={\rm \textbf{SigmaUpdate} \vspace{3pt}}(\bw^{r,\tau_{r}},w^{r,0},\cT_1,A)$, $r=r+1;$ \vspace{3pt}}
            \Until   termination criteria hold \vspace{3pt}
             \State{\textbf{Output:} $\{\bw^{r,t}\}$.}
		\end{algorithmic}
	\end{algorithm}

\subsection{Restart criteria}
From Theorem \ref{Th:complexity-acc-pADMM}, we know that the base algorithm, the HPR method with semi-proximal terms, achieves the iteration complexity of \(O(1/k)\) with respect to the KKT residual and the objective error. This result is derived from the complexity analysis in Proposition \ref{prop:complexity-acc-pADMM-M}, which offers a tighter bound for designing restart criteria. Thus, we define the following merit function:
\[
R_{r,t} := \|w^{r,t} - w^*\|_{\cM}, \quad \forall r \geq 0, \ t \geq 0,
\]
where \( w^* \) is an arbitrary solution to the KKT system \eqref{eq:KKT}. Note that \(R_{r,0}\) represents the upper bound (disregarding the factor 2) in the complexity results in Proposition \ref{prop:complexity-acc-pADMM-M} at the $r$-th outer loop. A natural strategy is to restart the inner loop when \( R_{r,t} \) has sufficiently decreased compared to \( R_{r,0} \), i.e., \( R_{r,t} \leq \alpha_1 R_{r,0} \), where \( \alpha_1 \in (0,1) \). Unfortunately, in practice, if \( \alpha_1 \) is too small, the algorithm is likely to fail to achieve this sufficient reduction. Hence, we also consider the length of the inner loop and the oscillation of the merit function \( R_{r,t}\). These ideas have been implemented in PDLP with different merit functions \cite{applegate2021practical,lu2023cupdlp,lu2023cupdlp_c}.\footnote{In \cite{applegate2021practical}, Applegate et al. used a normalized duality gap as the merit function, whereas Lu et al. selected a weighted KKT residual as the merit function in \cite{lu2023cupdlp,lu2023cupdlp_c}.} In addition,  since \( w^* \) is unknown, we approximate \( R_{r,t} \) using the following expression, inspired by Proposition \ref{prop:complexity-acc-pADMM-M}:
\[
\widetilde{R}_{r,t} = \|w^{r,t} - \hat{w}^{r,t+1}\|_{\cM}.
\]
Consequently, the restart criteria in HPR-LP are defined as follows:
\begin{enumerate}
    \item  {Sufficient decay of $\widetilde{R}_{r,t+1}$:}
    \begin{equation}\label{eq:restart_1}
            \widetilde{R}_{r,t+1}\leq \alpha_{1} \widetilde{R}_{r,0};
    \end{equation}
    \item  {Necessary decay + no local progress of  $\widetilde{R}_{r,t+1}$:}
    \begin{equation}\label{eq:restart_2}    \widetilde{R}_{r,t+1}\leq \alpha_{2} \widetilde{R}_{r,0}\quad  \text{ and }\quad   \widetilde{R}_{r,t+1} > \widetilde{R}_{r,t};
    \end{equation}
    \item  {Long inner loop:}
     \begin{equation}\label{eq:restart_3} t\geq \alpha_3k,
     \end{equation}
\end{enumerate}
where \( \alpha_1 \in (0, \alpha_2) \), \( \alpha_2 \in (0,1) \), and \( \alpha_3 \in (0,1) \). Once any of the three restart criteria is met, we restart the inner loop for the \((r+1)\)-th iteration, set \( w^{r+1,0} = \bar{w}^{r,\tau_r} \), and update \( \sigma_{r+1} \).

\subsection{Update rule for $\sigma$}
The update rule for $\sigma$ in HPR-LP is also derived from the complexity results of the HPR method with semi-proximal terms in Proposition \ref{prop:complexity-acc-pADMM-M}. Specifically, we update \(\sigma_{r+1}\) at the \((r+1)\)-th restart for any \(r \geq 0\) by solving the following optimization problem:
\begin{equation}\label{eq:sigma-0}
\sigma_{r+1} := \arg \min_{\sigma} \left\| w^{r+1,0} - w^* \right\|_{\mathcal{M}}^2,
\end{equation}
where \(\left\| w^{r+1,0} - w^* \right\|_{\mathcal{M}}\) represents the upper bound of the complexity results in Proposition \ref{prop:complexity-acc-pADMM-M} at the \((r+1)\)-th outer loop. A smaller upper bound is expected to lead to a smaller \(\|w^{r+1,t} - \hat{w}^{r+1,t+1}\|_{\cM}\) for any \(t \geq 0\), which further results in a smaller KKT residual \(\|\mathcal{R}(\bw^{r+1,t+1})\|\). Substituting the definition of \(\mathcal{M}\) from \eqref{def:M} into \eqref{eq:sigma-0}, we derive
\begin{equation}\label{eq:sigma}
 \begin{aligned}
\sigma_{r+1} &= \arg \min_{\sigma} \left\| w^{r+1,0} - w^* \right\|_{\mathcal{M}}^2 \\
&= \arg \min_{\sigma} \left( \sigma \| y^{r+1,0} - y^* \|^2_{\mathcal{T}_1} +
\sigma^{-1}
\| x^{r+1,0} - x^* + \sigma A^*(y^{r+1,0} - y^*) \|^2 \right) \\
&= \sqrt{ \frac{\| x^{r+1,0} - x^* \|^2}{\| y^{r+1,0} - y^* \|_{\mathcal{T}_1}^2 + \| A^*(y^{r+1,0} - y^*) \|^2} }.
\end{aligned}
\end{equation}
Since computing \(\| x^{r+1,0} - x^* \|\) and \(\| y^{r+1,0} - y^* \|_{\mathcal{T}_1}^2 + \| A^*(y^{r+1,0} - y^*) \|^2\) are not implementable, we approximate these terms in HPR-LP using
\begin{equation}\label{eq:Dx&Dy}
\displaystyle \Delta_x := \| \bx^{r,\tau_r} - x^{r,0} \| \; \;  \text{and} \; \; \Delta_y := \sqrt{\| \by^{r,\tau_r} - y^{r,0} \|_{\mathcal{T}_1}^2 + \| A^*(\by^{r,\tau_r} - y^{r,0}) \|^2},
\end{equation}
 respectively. Consequently, we update \(\sigma_{r+1}\) as follows:
\begin{equation}\label{eq:sigma_approx}
\sigma_{r+1} = \frac{\Delta_x}{\Delta_y}.
\end{equation}
Note that the approximations \(\Delta_x\) and \(\Delta_y\) may deviate significantly from the true values, so we update \(\sigma\) using formula \eqref{eq:sigma_approx} only when the following conditions are met; otherwise, we reset \(\sigma\) to 1:
\begin{enumerate}
    \item \(\Delta_x\) and \(\Delta_y\) are within a suitable range:
    \begin{equation}\label{eq:Dx_Dy}
          \Delta_x\in (10^{-16},10^{12}), \quad \Delta_y\in (10^{-16},10^{12});
    \end{equation}
    \item The relative infeasibilities of primal and dual problems should not differ excessively:
    \begin{equation}\label{eq:primal_dual_infeasibility}
\frac{{\rm error}_{d}}{{\rm error}_{p}} \in (10^{-8},10^{8}),
 \end{equation}
 where
   $$
    {\rm error}_{p} := \frac{\|\Pi_{D}(b - A\bx^{r,\tau_r})\|}{1 + \|b\|}, \quad {\rm error}_{d} := \frac{\|c - A^{*}\by^{r,\tau_r} - \bz^{r,\tau_r}\|}{1 + \|c\|}.
    $$
\end{enumerate}
In summary, the update rule for $\sigma$ is presented in Algorithm \ref{alg:sigma}.
    \begin{algorithm}[H]
	\caption{${\rm \textbf{SigmaUpdate}}$}
	\label{alg:sigma}
	\begin{algorithmic}[1]
    	\State {\textbf{Input:} $(\bw^{r,\tau_{r}},w^{r,0},\cT_1,A)$.}                     \State {Calculate $\Delta_x$ and $\Delta_y$ defined in \eqref{eq:Dx&Dy};}
            \If{ conditions \eqref{eq:Dx_Dy} and \eqref{eq:primal_dual_infeasibility} are satisfied}
            \State{$\displaystyle \sigma_{r+1}=\frac{\Delta_x}{\Delta_y};$}
            \Else
            \State{ $\sigma_{r+1}=1;$}
            \EndIf
            \State{\textbf{Output:} $\sigma_{r+1}$.}
	\end{algorithmic}
    \end{algorithm}
\begin{remark}

The update rule for \(\sigma\) is not only applicable to LP problems but can also be directly extended to the HPR method with semi-proximal terms \cite{sun2024accelerating} for solving more general convex optimization problems \eqref{primal}.
\end{remark}
We discuss two special choices of proximal operator $\cT_1$ to obtain the update formula of $\sigma$ as follows. For other choices of $\cT_1$, one can use \eqref{eq:sigma_approx} to determine $\sigma$.

\begin{enumerate}
    \item $\cT_1=0$. For example, consider the LP problems without inequality constraint (i.e., \(m_2=0\)). At the \(r\)-th outer loop and \(t\)-th inner loop, \(\by^{r,t+1}\) can be obtained by solving the following linear equations:
\begin{equation}\label{eq:AATy=R}
AA^{*}\by^{r,t+1} = \frac{1}{\sigma}(b - A(\bx^{r,t+1} + \sigma (\bz^{r,t+1} - c)) ).
\end{equation}
Solving the linear equations \eqref{eq:AATy=R} by the direct method is affordable in many applications such as in the optimal transport problem \cite{zhang2024hot} or the Wasserstein barycenter problem \cite{zhang2022efficient}. According to \eqref{eq:sigma_approx}, if conditions \eqref{eq:Dx_Dy} and \eqref{eq:primal_dual_infeasibility} are satisfied, then \(\sigma_{r+1}\) is computed as:
\begin{equation}\label{eq:sigma_T1=0}
\sigma_{r+1} = \frac{\|\bx^{r,\tau_r} - x^{r,0}\|}{\|A^{*}(\by^{r,\tau_r} - y^{r,0})\|}.
\end{equation}
\item $\mathcal{T}_1 = \lambda I_{m} - AA^{*}$ with $\lambda \geq \lambda_{1}(AA^{*})$ as proposed in \cite{esser2010general,chambolle2011first,xu2011class}. This applies when inequality constraints are present (i.e., \(m_2 > 0\)) or when solving the linear equations \eqref{eq:AATy=R} directly is not affordable. In this case,
\begin{equation}\label{update-y}
  \left\{
\begin{array}{ll}
\displaystyle \by_{1}^{r,t+1} = y^{r,t}_1+ \frac{1}{\lambda }(b_1/\sigma - A_1R_{y}),\\
\by_{2}^{r,t+1} = \Pi_{\mathbb{R}^{m_2}_+}\left(y^{r,t}_2+ \frac{1}{\lambda }(b_2/\sigma - A_2R_{y})\right),\\
\end{array}
\right.  
\end{equation}
where $R_{y}:=\bx^{r,t+1}/\sigma + (A^{*}y^{r,t}+ \bz^{r,t+1} - c)$. If conditions \eqref{eq:Dx_Dy} and \eqref{eq:primal_dual_infeasibility} are satisfied, then \(\sigma_{r+1}\) is given by:
\begin{equation}\label{eq:sigma_T1!=0}
\sigma_{r+1}=\frac{1}{\sqrt{\lambda}}\frac{\|\bx^{r,\tau_r} - x^{r,0}\|}{\|\by^{r,\tau_r} - y^{r,0}\|}.
\end{equation}
\end{enumerate}
\begin{remark}
The formula \eqref{eq:sigma_T1!=0} is similar to the primal weight \(\omega\) update formula with \(\theta = 1\) in Algorithm 3 of \cite{applegate2021practical}, except for the inclusion of the term \(\lambda\).
\end{remark}

{
\subsection{The GPU implementation of the HPR-LP}  

We first present the update formulas for each subproblem in HPR-LP (Steps 6-8). Specifically, for any \( r \geq 0 \) and \( t \geq 0 \), the update of \( \bz^{r,t+1} \) is given by:  

\begin{equation}\label{update-z}
\begin{array}{ll}
\bz^{r,t+1} &\displaystyle= \underset{z \in \mathbb{R}^{n}}{\arg \min }\left\{L_{\sigma_r}\left(y^{r,t}, z ; x^{r,t}\right)\right\}\\ 
&\displaystyle=  \frac{1}{\sigma_r} \Big( \Pi_C \big(x^{r,t}+\sigma_r(A^* y^{r,t}-c)\big) - \big(x^{r,t}+\sigma_r(A^* y^{r,t}-c)\big) \Big).
\end{array}
\end{equation}  
Next, the update of \( \bx^{r,t+1} \) is:  
\begin{equation}\label{update-x}
\bx^{r,t+1} = x^{r,t} + \sigma_r (A^{*} y^{r,t} + \bz^{r,t+1} - c)  
= \Pi_C\big(x^{r,t}+\sigma_r(A^* y^{r,t}-c)\big).
\end{equation}  
For general LP problems, we set \( \mathcal{T}_1 = \lambda I_{m} - AA^{*} \) with \( \lambda \geq \lambda_{1}(AA^{*}) \) in the HPR-LP. Consequently, following from \eqref{update-y} and \eqref{update-x}, the update for \( \by^{r,t+1} \) is given by:  

\begin{equation}\label{update-y-2}
\left\{
\begin{array}{ll}
\displaystyle \by_{1}^{r,t+1} = y^{r,t}_1+ \frac{1}{\lambda \sigma_r} \Big(b_1 - A_1(2\bx^{r,t+1}-x^{r,t})\Big),\\
\displaystyle \by_{2}^{r,t+1} = \Pi_{\mathbb{R}^{m_2}_+} \Big( y^{r,t}_2+ \frac{1}{\lambda \sigma_r} \Big(b_2 - A_2(2\bx^{r,t+1}-x^{r,t}) \Big) \Big).
\end{array}
\right.  
\end{equation}  
By combining \eqref{update-z}, \eqref{update-x}, and \eqref{update-y-2}, we observe that it is not necessary to compute \( \bz^{r,t+1} \) at every iteration. Instead, we only need to evaluate \( \bz^{r,t+1} \) using \eqref{update-z} when checking the termination criteria for stopping HPR-LP. Furthermore, the update of \eqref{update-z}, \eqref{update-x}, and \eqref{update-y-2} mainly involves the matrix-vector multiplications, vector additions, and projections. The overall per-iteration complexity of HPR-LP is $O(\operatorname{nnz}(A))$, where $\operatorname{nnz}(A)$  denotes the number of nonzero entries in $A$. 

Moreover, to fully utilize the parallel computing power of GPUs, we implement custom CUDA kernels for \eqref{update-z}, \eqref{update-x}, and \eqref{update-y-2}. For matrix-vector multiplications, we utilize \texttt{cusparseSpMV()} from the cuSPARSE library, which applies the \texttt{CUSPARSE\_SPMV\_CSR\_ALG2} algorithm to ensure deterministic results.

}

\section{Numerical experiments}\label{sec:4}
In this section, we compare the performance of HPR-LP implemented in Julia \cite{bezanson2017julia} with cuPDLP.jl \cite{lu2023cupdlp} on a GPU. The experimental setup is detailed in Section 4.1. Section 4.2 discusses the performance of these two solvers on Mittelmann’s LP benchmark set.\footnote{\href{https://plato.asu.edu/ftp/lpfeas.html}{https://plato.asu.edu/ftp/lpfeas.html}.} Section 4.3 presents numerical results on LP relaxations of instances from the MIPLIB 2017 collection \cite{gleixner2021miplib}. Finally, Section 4.4 highlights the numerical results on extremely large instances, including LPs generated from quadratic assignment problems (QAPs) \cite{burkard1997qaplib}, the ``zib03" instance as discussed by Koch et al. \cite{koch2022progress}, { and the LP formulation of the PageRank problem \cite{nesterov2014subgradient}}.

\subsection{Experimental setup}

\textbf{Benchmark datasets.} We conduct extensive performance experiments on both classical benchmark sets and extremely large-scale instances. Specifically, we evaluate HPR-LP and cuPDLP.jl \cite{lu2023cupdlp} on two classic benchmark datasets: Mittelmann's LP benchmark set and LP relaxations of instances from the MIPLIB 2017 collection. We use 49 publicly available Mittelmann’s LP benchmark instances. For the MIPLIB 2017 set, we initially select 383 instances based on the criteria outlined in \cite{lu2023cupdlp}. After excluding three instances during Gurobi's presolve \cite{gurobi}, we retain 380 instances.\footnote{Two instances are identified as unbounded by Gurobi's presolve, and one instance is solved by Gurobi's presolve.} We then split the MIP relaxations into three classes based on the number of nonzeros in the constraint matrix, as shown in Table \ref{tab:claases-MIP_R}, following the approach in \cite{lu2023cupdlp}. In addition, to further explore the limits of the capability of HPR-LP, we employ several extremely large-scale LP problems referenced in \cite{lu2023cupdlp_c}, including LPs generated from QAPs \cite{burkard1997qaplib} the ``zib03" instance from \cite{koch2022progress}, {and several instances of PageRank problem \cite{nesterov2014subgradient}}.
\begin{table}[ht!]
\centering
\caption{ Scales of instances in MIP relaxations.}
\label{tab:claases-MIP_R}
\renewcommand{\arraystretch}{1.5}
{
\begin{tabular}{cccc}
\hline
                    & Small   & Medium & Large \\ \hline
Number  of nonzeros & 100K-1M & 1M-10M & $>$10M  \\
Number of instances & 268     & 94     & 18    \\ \hline
\end{tabular}%
}
\end{table}

~\\
\noindent\textbf{Software and computing environment.} HPR-LP is implemented in Julia \cite{bezanson2017julia}, referred to as HPR-LP.jl. Similarly, cuPDLP.jl \cite{lu2023cupdlp}, the GPU version of PDLP, is also implemented in Julia.\footnote{We downloaded the cuPDLP.jl from \href{https://github.com/jinwen-yang/cuPDLP.jl}{https://github.com/jinwen-yang/cuPDLP.jl} on July 24th, 2024. The infeasibility detection function of cuPDLP.jl is disabled.} All tested solvers are run on an NVIDIA A100-SXM4-80GB GPU with CUDA 12.3, and the experiments are conducted in Julia 1.10.4 on Ubuntu 22.04.3 LTS.

~\\
\noindent\textbf{Presolve and preconditioning.} To assess the impact of presolve on the FOMs, we compare all tested algorithms on two sets of instances: the original instances without presolve and the instances with presolve by Gurobi 11.0.3 (academic license). For preconditioning in HPR-LP, we perform 10 iterations of Ruiz scaling \cite{ruiz2001scaling}, followed by the bidiagonal preconditioning used by Pock and Chambolle \cite{pock2011diagonal} with \(\alpha = 1\). Finally, we normalize \(b\) and \(c\) divided by \(\|b\| + 1\) and \(\|c\| + 1\), respectively. For cuPDLP.jl \cite{lu2023cupdlp}, we use the default preconditioning settings.

~\\
\noindent\textbf{Initialization and parameter setting.} We initialize HPR-LP with the {origin} point and set the penalty parameter \(\sigma_0 = 1\).  After preconditioning, we choose $\cT_{1}=\lambda_{1}(AA^{*})I_{m}-AA^{*}$ and compute \(\lambda_{1}(AA^{*})\) using the power method \cite{golub2013matrix}. {In HPR-LP, the restart criteria are based on conditions \eqref{eq:restart_1}, \eqref{eq:restart_2}, and \eqref{eq:restart_3}, with parameters  $\alpha_1 = 0.2$,  $\alpha_2 = 0.6$ , and  $\alpha_3 = 0.2$. The penalty parameter  $\sigma$  is updated using formula \eqref{eq:sigma_T1!=0}.}
For cuPDLP.jl \cite{lu2023cupdlp}, the default settings are applied. 

~\\
\noindent\textbf{Termination criteria.} In HPR-LP, the sequence $\{\bw^{r,t}\}$ is used to check the stopping criteria, which automatically satisfies $\bx^{r,t} \in C$ and $\by^{r,t} \in D$ for any $r\geq0,$ and $ t\geq1$. We terminate HPR-LP when the following stopping criteria used in PDLP \cite{applegate2021practical,lu2023cupdlp,lu2023cupdlp_c} are satisfied for the tolerance \(\varepsilon \in (0,\infty)\):
\[
\begin{aligned}
&\left| \langle b, y\rangle -\delta^{*}_{C}(-z) - \langle c, x \rangle \right| \leq \varepsilon \left( 1 + \left| \langle b, y\rangle - \delta^{*}_{C}(-z) \right| + \left| \langle c, x \rangle \right| \right), \\
&\|\Pi_{D}(b - Ax)\| \leq \varepsilon \left( 1 + \| b \| \right), \\
&\left\| c - A^{*} y - z \right\| \leq \varepsilon \left( 1 + \| c \| \right).
\end{aligned}
\]
We evaluate the performance of all tested algorithms with $\varepsilon=10^{-4}$, $10^{-6},$ and $10^{-8}$ for all the datasets. Moreover, in HPR-LP.jl, we check the termination and restart criteria every 150 iterations, {whereas cuPDLP.jl uses its default setting of checking termination criteria every 64 iterations.}

~\\
\noindent\textbf{Time limit.} In Section 4.3, we impose 15000 seconds as the time limit in Mittelmann’s benchmark. For Section 4.2, we impose a time limit of 3600 seconds on instances with small-sized and medium-sized instances, and a time limit of 18000 seconds for large instances. For Section 4.4, we impose {18000} seconds as the time limit for the LP instances generated by QAPs {and PageRank problems}, and 36000 seconds for the ``zib03" instance.
~\\

\noindent\textbf{Shifted geometric mean.} We use the shifted geometric mean of solving time, as employed in Mittelmann's benchmarks, to measure the performance of solvers on a collection of problems. Specifically, the shifted geometric mean is defined as
\[
(\prod_{i=1}^n\left(t_i+\Delta\right))^{1 / n}-\Delta,
\]
where \(t_i\) is the solving time in seconds for the \(i\)-th instance. We shift by \(\Delta = 10\) and denote this measure as SGM10. If an instance remains unsolved, the solving time is set to the corresponding time limit.  As with cuPDLP.jl \cite{lu2023cupdlp}, we exclude the time required for reading data, presolving, and preconditioning. For cuPDLP.jl, we use the default timing settings. For HPR-LP.jl, in addition to the algorithm runtime, we also include the time spent on the power method.

\subsection{Mittelmann’s LP benchmark set}
Mittelmann’s LP benchmark set is a standard benchmark for evaluating the numerical performance of different LP solvers. In this experiment, {we first analyze the impact of the restart strategy, the Halpern iteration, the update rule for \(\sigma\), and the relaxation step (Step 9 in Algorithm \ref{alg:HPR_LP}).} We then compare HPR-LP.jl with cuPDLP.jl on Mittelmann's LP benchmark set. 

{
Figure \ref{Fig:impact} summarizes the relative impact of HPR-LP's enhancements on 49 instances of Mittelmann’s LP benchmark set without presolve at a tolerance of \(10^{-8}\). As shown in subtable (a) of Figure \ref{Fig:impact}, the enhancements described in Section 3 significantly improve the performance of the baseline Douglas-Rachford (DR) method, which corresponds to Steps 1–3 in Algorithm \ref{alg:sp-HPR}. Meanwhile, subfigure (b) of Figure \ref{Fig:impact} presents the normalized SGM10 values, computed relative to the baseline DR method, offering a clear visual representation of performance gains. Specifically, incorporating restart criteria and the Halpern iteration, the Halpern DR method for LP (HDR-LP) with a fixed \(\sigma\) solves 4 more problems than the DR method and achieves a \textbf{2.30×} speedup in terms of SGM10. Furthermore, introducing the adaptive update rule for \(\sigma\) allows HDR-LP to solve 9 more problems than HDR-LP with a fixed \(\sigma\), resulting in a \textbf{5.20×} speedup in terms of SGM10. Finally, by incorporating an additional relaxation step (Step 9 in Algorithm \ref{alg:HPR_LP}), HPR-LP benefits from a larger step size than HDR-LP, leading to a \textbf{1.67×} speedup over HDR-LP in terms of SGM10.

\begin{figure}[H]
    \centering
    \begin{subfigure}[b]{0.9\textwidth}
        \centering
          \caption{SGM10 and the number of solved problems.}
          \renewcommand{\arraystretch}{1.5} 
        \begin{tabular}{l rr}
            \toprule
            Enhancements                  & Solved & SGM10 \\ 
            \midrule
            DR     & 27     & 2087.3     \\ 
           + restart \& Halpern (=HDR-LP with fixed $\sigma$) & 31     & 905.7      \\ 
           + $\sigma$ update  (=HDR-LP)                 & 42     & 174.1      \\ 
           + relaxation (=HPR-LP)                  & 43     & 103.8      \\ 
            \bottomrule
        \end{tabular}
    \end{subfigure}
    
    \hfill
    \begin{subfigure}[b]{0.9\textwidth}
        \centering
        \includegraphics[width=\textwidth]{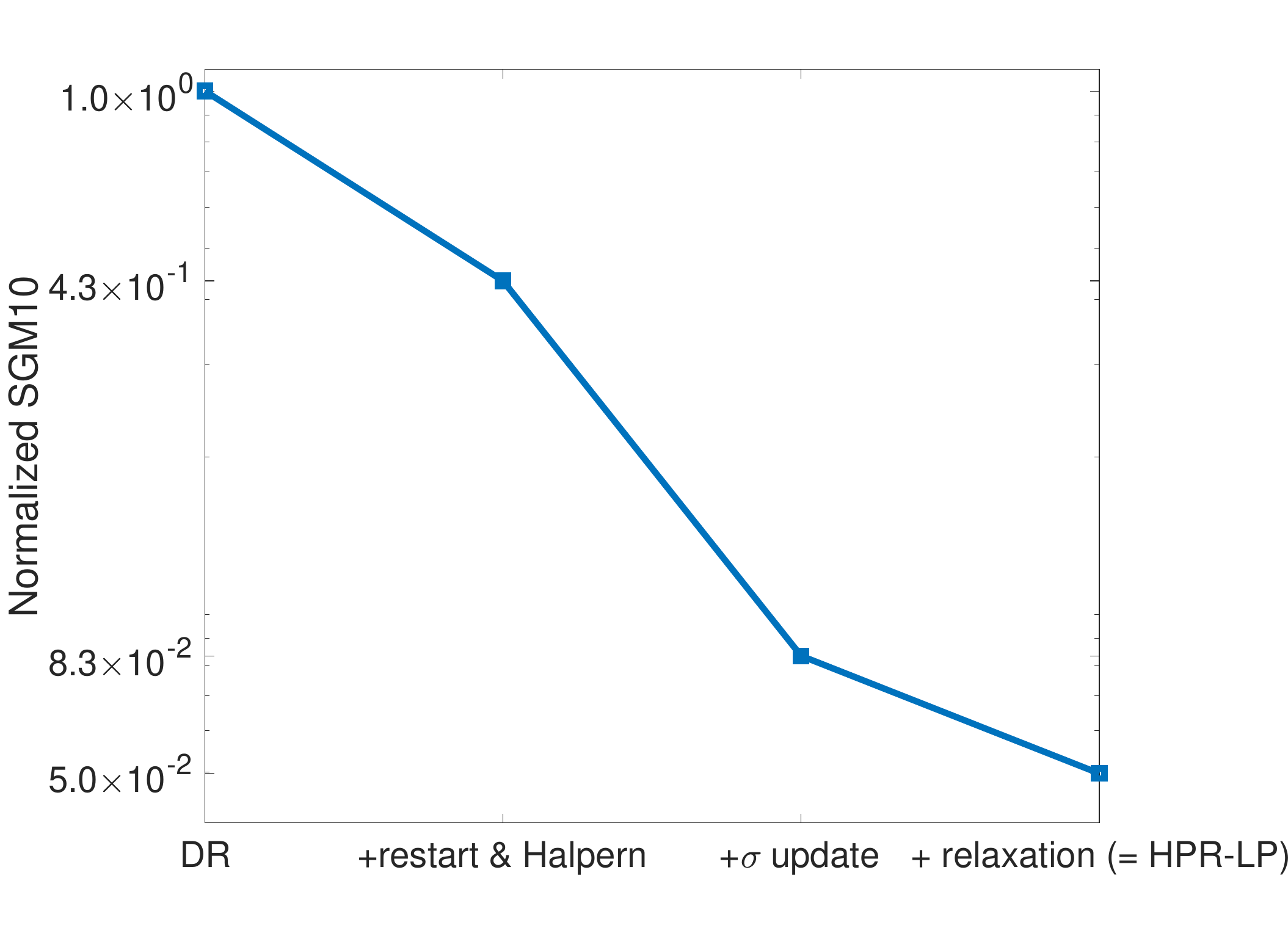}
        \caption{Normalized SGM10.}
    \end{subfigure}
    \caption{Relative impact of HPR-LP's enhancements on Mittelmann’s LP benchmark set without presolve.}
    \label{Fig:impact}
\end{figure}

}

The numerical performance of HPR-LP.jl and cuPDLP.jl, both with and without presolve, is presented in Tables \ref{tab: Hans-presolved} and \ref{tab: Hans-unpresolved}, respectively. The main results are summarized as follows:
\begin{itemize}[label=$\bullet$]
\item HPR-LP.jl consistently solves \textbf{3} to \textbf{5} more problems than cuPDLP.jl does across all tolerance levels, as shown in Tables \ref{tab: Hans-presolved} and \ref{tab: Hans-unpresolved}.
\item In terms of SGM10, HPR-LP.jl outperforms cuPDLP.jl on Mittelmann's LP benchmark set across all tolerance levels. For instance, as shown in Table \ref{tab: Hans-presolved}, HPR-LP.jl achieves a \textbf{3.71x} speedup over cuPDLP.jl to obtain a solution with a \(10^{-8}\) relative accuracy for the presolved dataset. Similarly, in Table \ref{tab: Hans-unpresolved}, HPR-LP.jl achieves a \textbf{2.68x} speedup over cuPDLP.jl to attain the same tolerance for the unpresolved dataset.
\item HPR-LP.jl exhibits better performance on the presolved dataset compared to the unpresolved dataset, since HPR-LP.jl solves more problems on the presolved Mittelmann’s LP benchmark set across all tolerance levels.
\end{itemize}

\begin{table}[H]
\centering
\caption{Numerical performance of different solvers on 49 instances of Mittelmann's LP benchmark set with presolve.}
\label{tab: Hans-presolved}
\renewcommand{\arraystretch}{1.5} 
\resizebox{\textwidth}{!}{%
\begin{tabular}{ccccccc}
\toprule
Tolerance & \multicolumn{2}{c}{$10^{-4}$} & \multicolumn{2}{c}{$10^{-6}$} & \multicolumn{2}{c}{$10^{-8}$} \\
\midrule
Solvers   & SGM10 & Solved & SGM10 & Solved & SGM10 & Solved \\
\midrule
cuPDLP.jl & 60.0  & 46     & 118.6      & 45      & 220.6 & 43    \\
HPR-LP.jl & 17.4   & 49     & 31.8      & 49      & 59.4  &  48    \\
\bottomrule
\end{tabular}
}
\end{table}
\begin{table}[ht!]
\centering
\caption{Numerical performance of different solvers on 49 instances of Mittelmann's LP benchmark set without presolve.}
\label{tab: Hans-unpresolved}
\renewcommand{\arraystretch}{1.5} 
\resizebox{\textwidth}{!}{%
\begin{tabular}{c c c c c c c}
\toprule
Tolerance & \multicolumn{2}{c}{$10^{-4}$} & \multicolumn{2}{c}{$10^{-6}$} & \multicolumn{2}{c}{$10^{-8}$} \\
\midrule
Solvers   & SGM10 & Solved & SGM10 & Solved & SGM10 & Solved \\
\midrule
cuPDLP.jl & 76.9   & 42     & 156.2  & 41     & 277.9  & 40     \\
HPR-LP.jl & 30.2   & 47     & 69.1   & 44      & 103.8  & 43     \\
\bottomrule
\end{tabular}
}
\end{table}

{
To further compare HPR-LP.jl with cuPDLP.jl, we present the shifted geometric mean of iteration counts for different solvers on 49 instances of Mittelmann’s LP benchmark set in Table \ref{tab:Hans-merged}, similar to the SGM10 metric for time. As shown in Table \ref{tab:Hans-merged}, HPR-LP.jl requires fewer iterations than cuPDLP.jl across different tolerance levels. Additionally, Table \ref{tab:per-iteration} reports the per-iteration time of both solvers. On average, the per-iteration time of cuPDLP.jl is \textbf{2.59×} that of HPR-LP.jl with presolve and \textbf{2.22×} without presolve. One possible reason is that cuPDLP.jl employs a line search strategy to achieve a larger step size, which incurs significant computational overhead \cite{lu2023cupdlp}. In contrast, HPR-LP.jl makes use of the theoretical advance in \cite{sun2024accelerating} to choose the desirable relaxation parameter value of 2 (Step 9 in Algorithm 2) to obtain a large step size, providing a guaranteed theoretical foundation for simpler and more efficient implementation.

\begin{table}[H]
\centering
\caption{Shifted geometric mean of iteration counts for different solvers on 49 instances of Mittelmann’s LP benchmark set, with and without presolve.}
\label{tab:Hans-merged}
\renewcommand{\arraystretch}{1.5} 
\resizebox{\textwidth}{!}{%
\begin{tabular}{c rr r rr r}
\toprule
& \multicolumn{3}{c}{With Presolve} & \multicolumn{3}{c}{Without Presolve} \\ 
\cmidrule(lr){2-4} \cmidrule(lr){5-7}
Solver & $10^{-4}$ & $10^{-6}$ & $10^{-8}$ & $10^{-4}$ & $10^{-6}$ & $10^{-8}$ \\ 
\midrule
cuPDLP.jl & 47694.3  & 120748.1  & 253827.3  & 42997.0  & 125972.9  & 260003.4  \\ 
HPR-LP.jl & 27661.1  & 71380.7   & 150375.0  & 36018.2  & 115121.2  & 206069.4  \\ 
\bottomrule
\end{tabular}
}
\end{table}

\begin{table}[H]
\centering
\caption{Per-iteration time in seconds of different solves for Mittelmann’s LP benchmark set.}

\renewcommand{\arraystretch}{1.5} 
\resizebox{\textwidth}{!}{%

\label{tab:per-iteration}
\begin{tabular}{l rr rr}
\toprule
& \multicolumn{2}{c}{With Presolve} & \multicolumn{2}{c}{Without Presolve} \\ 
\cmidrule(lr){2-3} \cmidrule(lr){4-5}
Metric & HPR-LP.jl & cuPDLP.jl & HPR-LP.jl & cuPDLP.jl \\ 
\midrule
Median              & 1.4e-4  & 4.3e-4  & 2.0e-4  & 5.2e-4  \\ 
Mean                & 3.7e-4  & 9.6e-4  & 5.4e-4  & 1.2e-3  \\ 
Standard Deviation  & 9.0e-4  & 1.7e-3  & 7.2e-5  & 3.1e-4  \\ 
Min                 & 7.7e-5  & 3.0e-4  & 6.2e-3  & 1.1e-2  \\ 
Max                 & 6.2e-3  & 1.1e-2  & 1.1e-3  & 2.0e-3  \\ 
\bottomrule
\end{tabular}
}
\end{table}

}

\subsection{MIP relaxations}
In this experiment, we compare HPR-LP.jl with cuPDLP.jl on the MIP relaxations set, which includes 380 instances. The numerical performance of these two solvers, both with and without presolve, is presented in Tables \ref{tab:MIP_presolved} and \ref{tab:MIP_unpresolved}, respectively. The main results are summarized as follows:
\begin{itemize}[label=$\bullet$]
\item With a \(10^{-8}\) accuracy, HPR-LP.jl solves \textbf{7} more problems than cuPDLP.jl does across the 380 presolved MIP relaxation instances, as shown in Table \ref{tab:MIP_presolved}. Moreover, on the unpresolved dataset, HPR-LP.jl solves more problems than cuPDLP.jl does in different stopping tolerances, as shown in Table \ref{tab:MIP_unpresolved}.
\item In terms of SGM10 across 380 instances, HPR-LP.jl consistently outperforms cuPDLP.jl at all tolerance levels. As shown in Table \ref{tab:MIP_presolved}, HPR-LP.jl achieves a \textbf{2.39x} speedup over cuPDLP.jl to obtain a solution with a \(10^{-8}\) accuracy for the presolved dataset. Similarly, in Table \ref{tab:MIP_unpresolved}, HPR-LP.jl achieves a \textbf{2.03x} speedup over cuPDLP.jl to obtain a solution with the same accuracy for the unpresolved dataset.
\end{itemize}

\begin{table}[H]
\centering
\caption{Numerical performance of different solvers on 380 instances of MIP relaxations with presolve.}
\label{tab:MIP_presolved}
\renewcommand{\arraystretch}{1.5} 
\resizebox{\textwidth}{!}{%
\begin{tabular}{c c c c c c c c}
\toprule
\multirow{2}{*}{Classes} & Tolerance & \multicolumn{2}{c}{$10^{-4}$} & \multicolumn{2}{c}{$10^{-6}$} & \multicolumn{2}{c}{$10^{-8}$} \\
\cmidrule(lr){2-8}
                         & Solvers   & SGM10 & Solved & SGM10 & Solved & SGM10 & Solved \\
\midrule
\multirow{2}{*}{Small (268)} & cuPDLP.jl & 8.4    & 262    & 15.3      & 261      & 23.7   & 254    \\
                             & HPR-LP.jl & 4.1    & 262    & 6.3      & 261      & 8.9    & 261    \\
\midrule
\multirow{2}{*}{Medium (94)} & cuPDLP.jl & 10.4   & 94      & 23.6      & 92      & 35.4   & 92     \\
                             & HPR-LP.jl & 5.8    & 94     & 11.0      & 92      & 15.8   & 92     \\
\midrule
\multirow{2}{*}{Large (18)}  & cuPDLP.jl & 31.7   & 17     & 64.1      & 17      & 102.4  & 17     \\
                             & HPR-LP.jl & 22.9   & 17     & 44.0      & 17      & 76.0   & 17     \\
\midrule
\multirow{2}{*}{Total (380)} & cuPDLP.jl & 9.6    & 373    & 18.6      & 370      & 28.4   & 363    \\
                             & HPR-LP.jl & 5.1    & 373    & 8.3      & 370      & 11.9   & 370    \\
\bottomrule
\end{tabular}
}
\end{table}
\begin{table}[H]
\centering
\caption{Numerical performance of different solvers on 380 instances of MIP relaxations without presolve.}
\label{tab:MIP_unpresolved}
\renewcommand{\arraystretch}{1.5} 
\resizebox{\textwidth}{!}{%
\begin{tabular}{c c c c c c c c}
\toprule
\multirow{2}{*}{Classes} & Tolerance & \multicolumn{2}{c}{$10^{-4}$} & \multicolumn{2}{c}{$10^{-6}$} & \multicolumn{2}{c}{$10^{-8}$} \\
\cmidrule(lr){2-8}
                         & Solvers   & SGM10 & Solved & SGM10 & Solved & SGM10 & Solved \\
\midrule
\multirow{2}{*}{Small (268)} & cuPDLP.jl & 11.5   & 266    & 20.0      & 262      & 27.7   & 258    \\
                             & HPR-LP.jl & 4.6    & 267    & 7.6      & 265      & 12.6   & 259    \\

\midrule
\multirow{2}{*}{Medium (94)} & cuPDLP.jl & 14.9   & 90     & 26.7      & 89      & 43.8   & 87     \\
                             & HPR-LP.jl & 7.4    & 92     & 13.7      & 91      & 19.8   & 90     \\

\midrule
\multirow{2}{*}{Large (18)}  & cuPDLP.jl & 129.8  & 16     & 253.3      & 15      & 442.2  & 14     \\
                             & HPR-LP.jl & 117.6  & 17     & 260.7      & 15      & 428.6  & 14     \\
\midrule
\multirow{2}{*}{Total (380)} & cuPDLP.jl & 14.3   & 372    & 25.0      & 366      & 36.3   & 359    \\
                             & HPR-LP.jl & 6.9    & 376    & 11.6      & 371      & 17.9   & 363    \\
\bottomrule
\end{tabular}
}
\end{table}

\subsection{{Large-scale applications: QAP,  ZIB problem, and PageRank instances}}
In this experiment, we evaluate the performance of HPR-LP.jl and cuPDLP.jl on extremely large-scale LP datasets. Specifically, we first apply Adams-Johnson linearization \cite{adams1994improved} to generate LP instances of quadratic assignment problems (QAPs) from QAPLIB \cite{burkard1997qaplib}, formulated as follows:
\begin{equation}
\begin{aligned}
\min_{x, s} & \quad \sum_{i, j} \sum_{k, l} a_{ik} b_{jl} s_{ijkl} \\
\text{s.t.} & \quad \sum_i s_{ijkl} = x_{kl}, & \quad j, k, l &= 1, \ldots, N, \\
            & \quad \sum_j s_{ijkl} = x_{kl}, & \quad i, k, l &= 1, \ldots, N, \\
            & \quad s_{ijkl} = s_{klij}, & \quad i, j, k, l &= 1, \ldots, N, \\
            & \quad s_{ijkl} \geq 0, & \quad i, j, k, l &= 1, \ldots, N, \\
            & \quad \sum_j x_{ij} = 1, & \quad i &= 1, \ldots, N, \\
            & \quad \sum_i x_{ij} = 1, & \quad j &= 1, \ldots, N, \\
            & \quad 0 \leq x_{ij} \leq 1, & \quad i, j &= 1, \ldots, N,
\end{aligned}
\end{equation}
where \((a_{ik})_{N \times N}\) represents the flow matrix, and \((b_{jl})_{N \times N}\) represents the distance matrix in the facility location application. The solving time and SGM10 for the two tested algorithms on 20 QAP instances, both with and without presolve, are presented in Tables \ref{tab:QAP_presolved} and \ref{tab:QAP_unpresolved}, respectively. Across all tolerance levels, HPR-LP.jl consistently outperforms cuPDLP.jl in terms of SGM10. For example, HPR-LP.jl achieves a \textbf{{5.70x}} speedup over cuPDLP.jl on the presolved dataset in terms of SGM10 to obtain a solution with a \(10^{-8}\) relative accuracy, as shown in Table \ref{tab:QAP_presolved}. On the unpresolved dataset, as shown in Table \ref{tab:QAP_unpresolved}, HPR-LP.jl achieves a \textbf{{2.57x}} speedup over cuPDLP.jl in terms of SGM10 for the same relative accuracy.

\begin{table}[H]
\centering
\caption{Solving time in seconds and SGM10 for different solvers on 20 QAP instances \cite{burkard1997qaplib} with presolve.}
\label{tab:QAP_presolved}
\renewcommand{\arraystretch}{1.5} 
\resizebox{\textwidth}{!}{%
\begin{tabular}{ccccccc}
\toprule
Tolerance & \multicolumn{2}{c}{$10^{-4}$} & \multicolumn{2}{c}{$10^{-6}$} & \multicolumn{2}{c}{$10^{-8}$} \\ \midrule
Solver    & HPR-LP.jl     & cuPDLP.jl     & HPR-LP.jl     & cuPDLP.jl     & HPR-LP.jl     & cuPDLP.jl     \\ \midrule
esc64a    & 5.8           & 5.4           & 6.2           & 7.0           & 6.9           & 14.6          \\
lipa40a   & 0.6           & 5.5           & 3.5           & 35.2          & 19.3          & 240.2         \\
lipa40b   & 0.5           & 3.2           & 2.5           & 11.4          & 13.7          & 93.8          \\
lipa50a   & 6.2           & 11.2          & 8.2           & 72.7          & 38.5          & 365.3         \\
lipa50b   & 0.9           & 7.0           & 5.8           & 20.3          & 32.6          & 172.1         \\
lipa60a   & 2.5           & 20.8          & 12.4          & 141.8         & 85.5          & 1348.7        \\
lipa60b   & 2.1           & 8.3           & 10.0          & 37.2          & 71.4          & 275.9         \\
lipa70a   & 4.5           & 40.2          & 25.5          & 304.7         & 168.5         & 2111.9        \\
lipa70b   & 4.8           & 15.0          & 22.0          & 76.0          & 129.3         & 571.2         \\
sko56     & 1.8           & 18.0          & 8.1           & 108.5         & 100.0         & 550.0         \\
sko64     & 3.1           & 28.6          & 12.3          & 183.4         & 125.0         & 1650.6        \\
tai40a    & 0.5           & 3.8           & 1.9           & 12.9          & 14.3          & 99.9          \\
tai40b    & 4.3           & 11.0          & 12.9          & 141.6         & 215.3         & 456.3         \\
tai50a    & 1.0           & 6.4           & 4.2           & 24.4          & 31.4          & 186.8         \\
tai50b    & 1.8           & 14.6          & 12.8          & 522.2         & 303.2         & 839.4         \\
tai60a    & 2.7           & 9.1           & 8.4           & 46.6          & 66.1          & 344.3         \\
tai60b    & 18.9          & 160.6         & 48.9          & 311.5         & 1099.8        & 7641.9        \\
tai64c    & 2.1           & 4.9           & 2.1           & 5.6           & 2.6           & 5.5           \\
tho40     & 0.5           & 7.7           & 1.8           & 44.3          & 14.5          & 360.1         \\
wil50     & 1.3           & 5.9           & 4.1           & 34.3          & 56.1          & 341.2         \\ \midrule
SGM10     & 2.9           & 12.7          & 8.8           & 60.0          & 60.2          & 343.1         \\ \bottomrule

\end{tabular}%
}
\end{table}

\begin{table}[H]
\centering
\caption{Solving time in seconds and SGM10 for different solvers on 20 QAP instances \cite{burkard1997qaplib} without presolve.}
\label{tab:QAP_unpresolved}
\renewcommand{\arraystretch}{1.5} 
\resizebox{\textwidth}{!}{%
\begin{tabular}{ccccccc}
\toprule
Tolerance & \multicolumn{2}{c}{$10^{-4}$} & \multicolumn{2}{c}{$10^{-6}$} & \multicolumn{2}{c}{$10^{-8}$} \\ \midrule
Solver    & HPR-LP.jl     & cuPDLP.jl     & HPR-LP.jl     & cuPDLP.jl     & HPR-LP.jl     & cuPDLP.jl     \\ \midrule
esc64a    & 7.6           & 7.8           & 8.6           & 9.0           & 10.0          & 13.9          \\
lipa40a   & 4.2           & 13.7          & 37.7          & 112.7         & 322.6         & 1601.5        \\
lipa40b   & 5.1           & 16.6          & 45.9          & 110.6         & 421.5         & 1317.4        \\
lipa50a   & 14.6          & 36.8          & 103.2         & 285.8         & 756.4         & 2446.9        \\
lipa50b   & 14.7          & 38.9          & 115.2         & 372.0         & 1187.3        & 4244.2        \\
lipa60a   & 27.2          & 95.5          & 218.2         & 479.7         & 1961.1        & 4556.0        \\
lipa60b   & 36.0          & 68.8          & 273.9         & 823.0         & 2642.0        & 7991.0        \\
lipa70a   & 54.4          & 170.0         & 461.0         & 1115.0        & 3639.9        & 10585.7       \\
lipa70b   & 75.9          & 154.5         & 534.7         & 1383.3        & 5287.8        & 15409.4       \\
sko56     & 32.6          & 55.7          & 515.6         & 1123.7        & 8511.9        & 15220.3       \\
sko64     & 37.5          & 85.1          & 357.0         & 994.5         & 4125.7        & 9047.3        \\
tai40a    & 5.8           & 12.7          & 45.2          & 163.0         & 426.9         & 1342.8        \\
tai40b    & 8.5           & 32.8          & 184.9         & 361.2         & 3898.6        & 13098.4       \\
tai50a    & 15.3          & 29.0          & 126.0         & 283.6         & 995.9         & 3322.1        \\
tai50b    & 16.4          & 62.5          & 449.5         & 836.4         & 8287.1        & 14482.2        \\
tai60a    & 35.3          & 76.8          & 272.5         & 811.0         & 2408.6        & 8713.1        \\
tai60b    & 62.6          & 163.6         & 1939.7        & 1663.0        & 18000.0       & 18000.0       \\
tai64c    & 4.8           & 9.2           & 7.7           & 12.6          & 12.4          & 35.5          \\
tho40     & 5.0           & 15.2          & 39.4          & 141.7         & 613.3         & 1405.6        \\
wil50     & 16.4          & 43.7          & 170.8         & 491.2         & 1171.7        & 3581.3        \\ \midrule
SGM10     & 18.9          & 43.9          & 150.7         & 342.4         & 1246.4        & 3202.5        \\ \bottomrule
\end{tabular}%
}
\end{table}

The numerical results of the instance ``zib03" in Table \ref{tab:zib03} show that HPR-LP.jl achieves a \textbf{4.47x} speedup over cuPDLP.jl on the presolved dataset and a \textbf{4.06x} speedup on the unpresolved dataset, both in terms of SGM10, to return a solution with a \(10^{-8}\) relative accuracy. {Moreover, we observe that for a relative accuracy of \(10^{-4}\), both HPR-LP.jl and cuPDLP.jl require more solution time with presolve than without. One possible explanation is that ``zib03" has been preprocessed by CPLEX 11, and excessive preprocessing may not always enhance algorithm performance. In some cases, it can even be detrimental by introducing numerical issues and disrupting useful problem structures. For instance, in this case, the coefficients of \( c \) originally range from \([1,2]\), but after presolve, they shift to \([1\text{e-}6, 3]\), potentially affecting numerical stability. 

}

\begin{table}[H]
\centering
\caption{Solving time in seconds for the ``zib03" instance \cite{koch2022progress}.}
\label{tab:zib03}
\renewcommand{\arraystretch}{1.5} 
\resizebox{\textwidth}{!}{%
\begin{threeparttable}
\begin{tabular}{c c c c c c c}
\toprule
Tolerance & \multicolumn{2}{c}{$10^{-4}$} & \multicolumn{2}{c}{$10^{-6}$} & \multicolumn{2}{c}{$10^{-8}$} \\
\midrule
Solver    & HPR-LP.jl     & cuPDLP.jl     & HPR-LP.jl     & cuPDLP.jl     & HPR-LP.jl     & cuPDLP.jl     \\
\midrule
With presolve\tnote{a}    & 273.8 & 351.9 & 1317.2 & 1634.6 & 3685.8 & 16462.2 \\
Without presolve\tnote{b} & 154.2 & 237.7 & 1063.6 & 1963.9 & 4865.3 & 19746.4 \\
\bottomrule
\end{tabular}
\begin{tablenotes}
\item[a] Although ``zib03" has been preprocessed by CPLEX 11, we still apply Gurobi’s presolve function to this data for consistency in the experimental setup. The matrix $A$ in ``zib03" contains 19,701,908 rows, 29,069,187 columns, and 104,300,584 non-zeros after presolve.
\item[b] The matrix $A$ in ``zib03" contains 19,731,970 rows, 29,128,799 columns, and 104,422,573 non-zeros without presolve.
\item[c] The commercial LP solver COPT used 16.5 hours to solve this instance on an AMD Ryzen 9 5900X \cite{lu2023cupdlp_c}.
\end{tablenotes}
\end{threeparttable}
}
\end{table}

{
Following the approach in \cite{applegate2021practical}, we generate multiple instances of the PageRank problem with varying numbers of nodes. In this experiment, the presolve procedure has minimal impact on these instances. Therefore, we focus on reporting the numerical results for all tested solvers on the original (unpresolved) dataset, as summarized in Table \ref{tab:pagerank_original}. Across all tolerance levels, {HPR-LP.jl} consistently outperforms {cuPDLP.jl} in terms of SGM10. For instance, on the unpresolved dataset, {HPR-LP.jl} achieves a \textbf{2.16$\times$} speedup over {cuPDLP.jl} when obtaining a solution with \({10^{-8}}\) relative accuracy, as shown in Table \ref{tab:pagerank_original}.
}

\begin{table}[H]
\centering
\caption{Solving time in seconds and SGM10 for different solvers on 4 PageRank instances without presolve.}
\label{tab:pagerank_original}
\renewcommand{\arraystretch}{1.5} 
\resizebox{\textwidth}{!}{%
\begin{tabular}{c rrr rrr}
\toprule
Tolerance     & \multicolumn{2}{c}{$10^{-4}$} & \multicolumn{2}{c}{$10^{-6}$} & \multicolumn{2}{c}{$10^{-8}$} \\ 
\cmidrule(lr){2-3} \cmidrule(lr){4-5} \cmidrule(lr){6-7}
Nodes       & HPR-LP.jl & cuPDLP.jl & HPR-LP.jl & cuPDLP.jl & HPR-LP.jl & cuPDLP.jl \\ 
\midrule
$10^{4}$ & 0.1  & 1.4  & 0.1  & 1.5  & 0.1  & 1.6  \\ 
$10^{5}$ & 0.1  & 1.4  & 0.1  & 1.5  & 0.2  & 1.8  \\ 
$10^{6}$ & 0.4  & 1.7  & 0.5  & 2.4  & 0.8  & 2.8  \\ 
$10^{7}$ & 9.9  & 7.5  & 15.2 & 32.6 & 20.7 & 46.6 \\ 
SGM10         & 2.0  & 2.8  & 2.8  & 6.3  & 3.6  & 7.8  \\ 
\bottomrule
\end{tabular}
}
\end{table}

\section{Conclusion}\label{sec:5}
In this paper, we developed an HPR-LP solver for solving large-scale LP problems, which integrates an adaptive strategy of restart and penalty parameter update into an HPR method with semi-proximal terms. Extensive numerical results on LP benchmark datasets showcased the efficiency of the HPR-LP solver in computing solutions varying from low to high accuracy. Nevertheless, there are still a few instances where the HPR-LP solver fails to obtain solutions within the allocated time. Exploring an accelerated sPADMM with a faster iteration complexity than $O(1/k)$ in terms of the KKT residual may be beneficial. One possible way to improve the performance of the HPR-LP solver is to incorporate the fast Krasnosel'skii-Mann iteration, which can yield an iteration complexity of \(o(1/k)\) in terms of the KKT residual and the objective error \cite{sun2024accelerating}.  Another possible way is to design a similar algorithm with a linear rate better than that possessed by sPADMM \cite{han2018linear}. We leave these as our future research directions.

\appendix
\section{The equivalence between the Halpern accelerating pADMM without proximal terms and the HPR method without proximal terms}\label{appendix-A}
Let $\mathbb{X}$, $\mathbb{Y}$, and $\mathbb{Z}$ be three finite-dimensional real Euclidean spaces, each equipped with an inner product $\langle \cdot, \cdot \rangle$ and its induced norm $\|\cdot\|$. In this appendix, we aim to establish the equivalence between the HPR method without proximal terms \cite{zhang2022efficient} and the Halpern accelerating pADMM without proximal terms \cite{sun2024accelerating} for solving the following convex optimization problem:
	\begin{equation}\label{primal}
		\begin{array}{cc}
			\min _{y \in \mathbb{Y} , z \in \mathbb{Z}} & f_{1}(y)+f_{2}(z)\\
			\text{subject to}& {B}_{1}y+{B}_{2}z=c,
		\end{array}
	\end{equation}
	where $f_{1}: \mathbb{Y} \to (-\infty, +\infty]$ and $f_{2}: \mathbb{Z} \to (-\infty, +\infty]$ are two proper closed convex functions, 
	$B_1: \mathbb{Y} \to \mathbb{X}$ and $B_2: \mathbb{Z} \to \mathbb{X}$ are two given linear operators, and $c\in \mathbb{X}$ is a given point. Let $\sigma >0$ be a given penalty parameter. The augmented Lagrangian function of problem \eqref{primal} is defined by, for any  $(y,z,x) \in \mathbb{Y}\times \mathbb{Z}\times \mathbb{X}$,
	\[
	L_{\sigma}(y, z; x) := f_1(y) + f_2(z) + \langle x, B_1 y + B_2 z - c \rangle+\frac{\sigma}{2}\|B_1 y + B_2 z - c\|^2.
	\]
	The {dual of problem   \eqref{primal} is given by}
	\begin{equation}\label{dual}
		\max _{x \in \mathbb{X}}\left\{-f_{1}^{*}(-B_{1}^* x)-f_{2}^{*}(-B_{2}^* x)-\langle c, x\rangle\right\},
	\end{equation}
	where $B_1^*: \mathbb{X} \to \mathbb{Y}$ and $B_2^*: \mathbb{X} \to \mathbb{Z}$ are the adjoint of $B_1$ and $B_2$, respectively. Subsequently, the HPR method without proximal terms \cite{zhang2022efficient} and the Halpern accelerating pADMM without proximal terms \cite{sun2024accelerating} for solving problem \eqref{primal} are detailed in Algorithm \ref{alg:HPR-OP} and Algorithm \ref{alg:acc-GADMM}, respectively.
    \begin{algorithm}[ht!]
	\caption{An HPR method without proximal terms \cite{zhang2022efficient} for solving convex optimization problem \eqref{primal}}
	\label{alg:HPR-OP}
	\begin{algorithmic}[1]
		\State {Input: $y^{0} \in {\rm dom}(f_1)$, $x^{0} \in \mathbb{X}$, and $\sigma>0$.}
       \State{Initialization: $\Tilde{x}^{0} := x^0$.}
        \For{$k=0,1,...,$}
       	\State {Step 1. $\displaystyle {{z}}^{k+1}= \underset{z\in \mathbb{Z}}{\arg \min }\left\{L_{\sigma}(y^k, z; \Tilde{x}^{k})\right\}$;}
		\State{Step 2. $\displaystyle x^{k+\frac{1}{2}}=\Tilde{x}^k+\sigma (B_{1}{y}^{k}+B_{2}z^{k+1}-c) $;}
		\State {Step 3. $\displaystyle y^{k+1} = \underset{y\in \mathbb{Y}}{\arg \min }\left\{L_{\sigma}(y, z^{k+1}; x^{k+\frac{1}{2}})\right\}$;}
		\State {Step 4. $\displaystyle x^{k+1}= {x}^{k+\frac{1}{2}}+\sigma (B_{1}{y}^{k+1}+B_{2}z^{k+1}-c)$;}		
		\State {Step 5. $\displaystyle \Tilde{x}^{k+1}= \left(\frac{1}{k+2} \Tilde{x}^{0} + \frac{k+1}{k+2} x^{k+1}\right) + \frac{\sigma}{k+2}\left[B_1{y}^{0}    -B_{1}{y}^{k+1}\right] $;}
         \EndFor
	\end{algorithmic}
\end{algorithm}

\begin{algorithm}[ht!]
		\caption{A Halpern accelerating pADMM without proximal terms \cite{sun2024accelerating} for solving convex optimization problem \eqref{primal}}
		\label{alg:acc-GADMM}
		\begin{algorithmic}[1]
			\State {Input: $w^{0}=(y^{0}, z^{0}, x^{0})\in \operatorname{dom} (f_1) \times \operatorname{dom} (f_2)\times \mathbb{X}$, and $\sigma > 0$. {Denote $w=(y,z,x)$ and $\bw=(\by,\bz,\bx)$.}}
           \For{$k=0,1,...,$}
   		\State {Step 1. $\displaystyle \bz^{k+1}=\underset{z\in \mathbb{Z}}{\arg \min }\left\{L_\sigma\left(y^k, z ; x^{k}\right)\right \}$;}		
			\State{Step 2. $\displaystyle \bx^{k+1}={x}^k+\sigma (B_{1}{y}^{k}+B_{2}\bz^{k+1}-c) $;}
            \State {Step 3. $\displaystyle \by^{k+1}=\underset{y \in \mathbb{Y}}{\arg \min }\left\{L_\sigma\left(y, \bz^{k+1} ; \bx^{k+1}\right)\right\}$;}
			\State {Step 4. $\displaystyle \hw^{k+1}= 2\bw^{k+1} -{w}^{k} $;}
			\State {Step 5. $\displaystyle w^{k+1}=\frac{1}{k+2}w^0+\frac{k+1}{k+2}\hw^{k+1}$;}
            \EndFor
		\end{algorithmic}
	\end{algorithm}

\begin{proposition}\label{prop:equiv-HPR} Assume that \(\partial f_1(\cdot) + \sigma B_{1}^{*}B_{1}\) and \(\partial f_2(\cdot) + \sigma B_{2}^{*}B_{2}\) are maximal and strongly monotone, respectively. Let \(w^{0} = (y^{0}, z^{0}, x^{0}) \in \operatorname{dom}(f_1) \times \operatorname{dom}(f_2) \times \mathbb{X}\). Then, the sequence \(\{(y^{k+1}, z^{k+1}, x^{k+\frac{1}{2}})\}\) generated by the HPR method without proximal terms in Algorithm \ref{alg:HPR-OP}, starting from the same initial point \((y^0, x^0) \in \operatorname{dom}(f_1) \times \mathbb{X}\), is equivalent to the sequence \(\{(\by^{k+1}, \bz^{k+1}, \bx^{k+1})\}\) produced by the Halpern accelerating pADMM without proximal terms in Algorithm \ref{alg:acc-GADMM}.
\end{proposition}
\begin{proof}
Based on the assumptions that \(\partial f_1(\cdot) + \sigma B_{1}^{*}B_{1}\) and \(\partial f_2(\cdot) + \sigma B_{2}^{*}B_{2}\) are both maximal and strongly monotone, we can conclude from \cite[Corollary 11.17]{bauschke2011convex} that the solution to each subproblem in both algorithms exists and is unique. We then prove the statement in the proposition using the method of induction. For \(k=0\), due to the same initial point \((y^{0}, x^{0})\), it is clear that:
\begin{equation}\label{eq:acc-yzx1}
  z^{1} = \bz^{1}, \quad x^{\frac{1}{2}} = \bx^{1}, \quad \text{and} \quad y^{1} = \by^{1}.
\end{equation}
Moreover, according to Steps 2, 4, and 5 in Algorithm \ref{alg:HPR-OP}, we know that \(\Tilde{x}^{1}\) satisfies:
\begin{equation}\label{eq:Tildex^1}
 \begin{array}{ll}
       \Tilde{x}^{1} &= \frac{1}{2}(\Tilde{x}^{0} + x^{1}) + \frac{\sigma}{2} B_1(y^0 - y^1) \\
       &= \frac{1}{2}(\Tilde{x}^{0} + x^{\frac{1}{2}} + \sigma(B_1 y^1 + B_2 z^1 - c)) + \frac{\sigma}{2} B_1(y^0 - y^1) \\
       &= \frac{1}{2}(x^{\frac{1}{2}} + \Tilde{x}^{0} + \sigma(B_1 y^0 + B_2 z^1 - c)) \\
       &= x^{\frac{1}{2}}.
 \end{array}
\end{equation}
It follows from Step 1 in Algorithm \ref{alg:HPR-OP} that \(z^{2}\) satisfies the following optimality condition:
\begin{equation}\label{eq:acc-y2-opt}
 \begin{array}{ll}
      & 0 \in \partial f_{2}(z^2) + B_{2}^{*} \Tilde{x}^{1} + \sigma B_{2}^{*}(B_{1} y^{1} + B_{2} z^{2} - c) \\
    \iff & 0 \in \partial f_{2}(z^2) + B_{2}^{*}(x^{\frac{1}{2}} + \sigma(B_{1} y^{1} + B_{2} z^{2} - c)).
 \end{array}
\end{equation}
On the other hand, \(\bz^{2}\) obtained by Step 1 in Algorithm \ref{alg:acc-GADMM} should satisfy the following optimality condition:
\begin{equation}\label{eq:bz2}
 \begin{array}{ll}
       0 \in \partial f_{2}(\bz^2) + B_{2}^{*}(x^1 + \sigma (B_{1} y^{1} + B_{2} \bz^{2} - c)).
 \end{array}
\end{equation}
Substituting
\begin{equation}\label{eq:x1&y1}
   x^{1} = \frac{1}{2} x^0 + \frac{1}{2}(2 \bx^1 - x^0) \quad \text{and} \quad y^{1} = \frac{1}{2} y^0 + \frac{1}{2}(2 \by^1 - y^0)
\end{equation}
from Steps 4 and 5 in Algorithm \ref{alg:acc-GADMM} into \eqref{eq:bz2}, we obtain:
\begin{equation*}
 \begin{array}{ll}
       0 \in \partial f_{2}(\bz^2) + B_{2}^{*}(\bx^1 + \sigma (B_{1} \by^{1} + B_{2} \bz^2 - c)).
 \end{array}
\end{equation*}
This, together with \eqref{eq:acc-yzx1} and \eqref{eq:acc-y2-opt}, implies:
\begin{equation}\label{eq:acc-z2}
z^2 = \bz^2.
\end{equation}
Furthermore, according to Step 2 in Algorithm \ref{alg:HPR-OP} and \eqref{eq:Tildex^1}, we know that:
\begin{equation}\label{eq:acc-x3/2-opt}
x^{\frac{3}{2}} = \Tilde{x}^{1} + \sigma (B_{1} y^{1} + B_{2} z^{2} - c) = x^{\frac{1}{2}} + \sigma (B_{1} y^{1} + B_{2} z^{2} - c).
\end{equation}
Also, from Step 2 in Algorithm \ref{alg:acc-GADMM} and \eqref{eq:x1&y1}, we know that \(\bx^{2}\) satisfies:
\begin{equation*}
\begin{array}{ll}
\bx^{2} &= x^1 + \sigma (B_{1} y^{1} + B_{2} \bz^{2} - c) \\
     &= \bx^1 + \sigma (B_1 \by^1 + B_2 \bz^2 - c).
\end{array}
\end{equation*}
It follows from \eqref{eq:acc-yzx1}, \eqref{eq:acc-z2}, and \eqref{eq:acc-x3/2-opt} that:
\begin{equation}\label{eq:acc-x3/2}
x^{\frac{3}{2}} = \bx^{2},
\end{equation}
which, together with \eqref{eq:acc-z2}, implies:
\begin{equation}\label{eq:acc-y2}
y^{2} = \by^{2}
\end{equation}
by comparing Step 3 in Algorithms \ref{alg:HPR-OP} and \ref{alg:acc-GADMM}. Hence, the statement is true for \(k=1\).

Now, assume that the statement holds for some $k\geq 1$. Then, by the assumption of induction, we have
\begin{equation}\label{eq:acc-yzxk+1}
y^{k+1}=\by^{k+1}, \quad z^{k+1}=\bz^{k+1}, \text{ and } x^{k+\frac{1}{2}}=\bx^{k+1}.
\end{equation}
According to Step 5 in Algorithm \ref{alg:HPR-OP} and \eqref{eq:acc-yzxk+1}, we know that $\Tilde{x}^{k+1}$ satisfies
\begin{equation}\label{eq:tildexk+1}
 \begin{array}{ll}
       \Tilde{x}^{k+1}&= \left(\frac{1}{k+2} \Tilde{x}^{0} + \frac{k+1}{k+2} x^{k+1}\right) + \frac{\sigma}{k+2}\left[B_1{y}^{0}    -B_{1}{y}^{k+1}\right]\\
       &= \frac{1}{k+2} \Tilde{x}^{0} + \frac{k+1}{k+2} (x^{k+1/2}+\sigma(B_1y^{k+1}+B_{2}z^{k+1}-c))  + \frac{\sigma}{k+2} B_1(y^0-y^{k+1})\\
      &= \frac{1}{k+2}x^0+\frac{k+1}{k+2}(\bx^{k+1}+\sigma(B_1 \by^{k+1}+B_2 \bz^{k+1}-c))+\frac{\sigma}{k+2}B_1(y^0-\by^{k+1})
       .\\
 \end{array}
\end{equation}
In addition, according to Step 1 in Algorithm \ref{alg:HPR-OP}, we know that $z^{k+2}$ satisfies
\begin{equation}\label{eq:acc-yk+2-opt}
 \begin{array}{ll}
& 0\in \partial f_{2}(z^{k+2})+ B_{2}^{*} \Tilde{x}^{k+1}   +\sigma B_{2}^{*}( B_{1}{y}^{k+1}+B_{2}z^{k+2}-c).
  \end{array}
\end{equation}
On the other hand, from Steps 2 and 5 in Algorithm \ref{alg:acc-GADMM}, we know that $x^{k+1}$ satisfies
\begin{equation}\label{eq:xk+1}
 \begin{array}{ll}
       x^{k+1}&= \frac{1}{k+2} x^{0} + \frac{k+1}{k+2} (2\bx^{k+1}-x^{k})\\
       &=\frac{1}{k+2} x^{0} + \frac{k+1}{k+2} (\bx^{k+1}+ \sigma (B_1 y^k+B_2\bz^{k+1}-c))
 \end{array}
\end{equation}
and
$y^{k+1}$ satisfies
\begin{equation}\label{eq:yk+1}
 \begin{array}{ll}
       y^{k+1}= \frac{1}{k+2} y^{0} + \frac{k+1}{k+2} (2\by^{k+1}-y^{k}).
 \end{array}
\end{equation}
In addition, from Step 1 in Algorithm \ref{alg:acc-GADMM}, we have
\begin{equation}\label{eq:bz^k+2}
    0\in \partial f_{2}(\bz^{k+2})+ B_{2}^{*}({x}^{k+1} +\sigma \left(B_{1}{y}^{k+1}+B_{2}\bz^{k+2}-c)\right).
\end{equation}
Substituting \eqref{eq:xk+1} and  \eqref{eq:yk+1} into \eqref{eq:bz^k+2}, we can obtain
$$
 \begin{array}{ll}
 0&\in \partial f_{2}(\bz^{k+2})+ B_{2}^{*}\left(\frac{1}{k+2}x^0+\frac{k+1}{k+2}(\bx^{k+1}+\sigma(B_1 \by^{k+1}+B_2 \bz^{k+1}-c))\right)\\
 & \quad +  B_{2}^{*}\left( \frac{\sigma}{k+2}B_1(y^0-\by^{k+1})+\sigma(B_1 \by^{k+1}+B_2\bz^{k+2}-c)
      \right),\\
 \end{array}
 $$
which together with \eqref{eq:tildexk+1} implies
 $$
 \begin{array}{ll}
 0\in \partial f_{2}(\bz^{k+2})+ B_{2}^{*}\left( \Tilde{x}^{k+1}+\sigma(B_1 \by^{k+1}+B_2\bz^{k+2}-c)
      \right).
 \end{array}
 $$
It follows from \eqref{eq:acc-yzxk+1} and \eqref{eq:acc-yk+2-opt} that
\begin{equation}\label{eq:acc-zk+2}
 \begin{array}{ll}
       z^{k+2}= \bz^{k+2}.
 \end{array}
\end{equation}
Moreover, according to Step 2 in Algorithm \ref{alg:HPR-OP}, we know that $x^{k+\frac{3}{2}}$ satisfies
\begin{equation}\label{eq:acc-xk+3/2-opt}
x^{k+\frac{3}{2}}= \Tilde{x}^{k+1}+\sigma (B_{1}{y}^{k+1}+B_{2}z^{k+2}-c).
\end{equation}
Also, from Step 2 in Algorithm \ref{alg:acc-GADMM}, \eqref{eq:tildexk+1}, \eqref{eq:xk+1} and  \eqref{eq:yk+1}, we know that $\bx^{k+2}$ satisfies
\begin{equation*}
\begin{array}{ll}
\bx^{k+2}&={x}^{k+1}+\sigma (B_{1}{y}^{k+1}+B_{2}\bz^{k+2}-c)\\
         &= \Tilde{x}^{k+1}+\sigma(B_1 \by^{k+1}+B_2\bz^{k+2}-c).
\end{array}
\end{equation*}
It follows from \eqref{eq:acc-yzxk+1}, \eqref{eq:acc-zk+2}, and \eqref{eq:acc-xk+3/2-opt} that
\begin{equation}\label{eq:acc-xk+3/2}
x^{k+\frac{3}{2}}=\bx^{k+2}.
\end{equation}
Since $(\bz^{k+2},\bx^{k+2})=(z^{k+2},x^{k+\frac{3}{2}})$, then we have
\begin{equation}\label{eq:acc-yk+2}
y^{k+2}=\by^{k+2}
\end{equation}
by comparing Step 3 in Algorithms \ref{alg:HPR-OP} and \ref{alg:acc-GADMM}. Hence, from \eqref{eq:acc-zk+2}, \eqref{eq:acc-xk+3/2}, and \eqref{eq:acc-yk+2}, we know that the statement also holds for $k+1$. Thus, by induction, we have completed the proof.
\qed
\end{proof}

%

\bibliography{REF.bib}

\bibliographystyle{spmpsci}      

\section*{Statements and Declarations}

\textbf{Funding} \thanks{The work of Defeng Sun was supported by the Research Center for Intelligent Operations Research, RGC Senior Research Fellow Scheme No. SRFS2223-5S02, and  GRF Project No. 15304721. The work of Yancheng Yuan was supported by the Research Center for Intelligent Operations Research and The Hong Kong Polytechnic University under grant P0045485. The work of Xinyuan Zhao was supported in part by the National Natural Science Foundation of China under Project No. 12271015.}

~\\
\noindent\textbf{Conflict of interest}   The authors declare that they have no conflict of interest.

~\\
\noindent\textbf{Data Availability} The datasets analyzed during the current study are available from the corresponding author on reasonable request.



\end{document}